\documentclass{amsart}

\usepackage{graphicx, psfrag, picins}

\newtheorem{theorem}{Theorem}[section]
\newtheorem{lemma}[theorem]{Lemma}
\newtheorem{corollary}[theorem]{Corollary}

\theoremstyle{definition}

\theoremstyle{remark}
\newtheorem{remark}[theorem]{Remark}

\numberwithin{equation}{section}

\begin{document}

\title{Rigidity of polyhedral surfaces, II}

\author{Ren Guo}

\address{Department of Mathematics, Rutgers University, Piscataway, NJ, 08854}

\email{renguo@math.rutgers.edu}

\author{Feng Luo}
\address{The Center of Mathematical Science, Zhejiang
University, Hangzhou, China, 310027}

\address{Department of Mathematics, Rutgers University, Piscataway, NJ, 08854}

\email{fluo@math.rutgers.edu}

\thanks{}



\begin{abstract} We study the rigidity of polyhedral surfaces using
variational principle. The action functionals are derived from the
cosine laws. The main focus of this paper is on the cosine law for
a non-triangular region bounded by three possibly disjoint
geodesics. Several  of these cosine laws were first discovered and
used by Fenchel and Nielsen. By studying the derivative of the
cosine laws, we discover a uniform approach on several variational
principles on polyhedral surfaces with or without boundary. As a
consequence, the work of Penner, Bobenko-Springborn and Thurston
on rigidity of polyhedral surfaces and circle patterns are
extended to a very general context.
\end{abstract}

\maketitle

\section {Introduction}
\subsection{Variational principle} We study geometry of polyhedral surfaces using
variational principles in this paper. This can be considered as a
continuation of the paper \cite{l4}. By a polyhedral surface we
mean an isometric gluing of geometric polygons in $\mathbb{E}^2$
(Euclidean plane), $\mathbb{H}^2$ (hyperbolic plane) or
$\mathbb{S}^2$ (the 2-sphere). We emphasize that the
combinatorics, i.e., the topological cellular decomposition
associated to a polyhedral surface, is considered as an intrinsic
part of the polyhedral surface. The investigation of the geometry
of polyhedral surface has a long history. Recent resurgence of
interests in this subject is mainly due to the work of William
Thurston on geometrization of 3-manifolds and circle packing on
surfaces since 1978. Thurston's and Andreev's work on circle
packing are non-variational. The variational approach to circle
packing was introduced in a seminal paper by Colin de Verdi\'ere
\cite{cdv}. Since then, many works on variational principle on
triangulated surfaces \cite{b}, \cite{r}, \cite{le}, \cite{bs},
\cite{l1}, \cite{l3} and others have appeared. A uniform approach
on variational principles on triangulated surfaces, based on the
derivative of the cosine law for triangles, was proposed in
\cite{l4}. It is shown in \cite{l4} that almost all known
variational principles on triangulated surfaces are covered by the
cosine law for triangles and right-angled hexagons and their
Legendre transformations.

The goal of this paper is to develop variational principles
arising from the cosine laws for hyperbolic polygons bounded by
three geodesics. Figure \ref{fig:triangles-in-klein-model} is the
list of all ten cases of triangles in the Klein model of the
hyperbolic plane. In Figure \ref{fig:ten-triangles}, generalized
hyperbolic triangles are drawn in the Poincar$\acute{e}$ model
where a horocycle is represented by a circle passing through a
vertex. Cosine laws for the cases $(1,1,-1), (-1,-1,1),
(-1,-1,-1)$ were discovered in Fenchel-Nielsen's work \cite{fn}.
R. Penner discovered the cosine law for the case $(0,0,0)$
(decorated ideal triangles) in \cite{p}.

\begin{center}
\includegraphics[scale=.35]{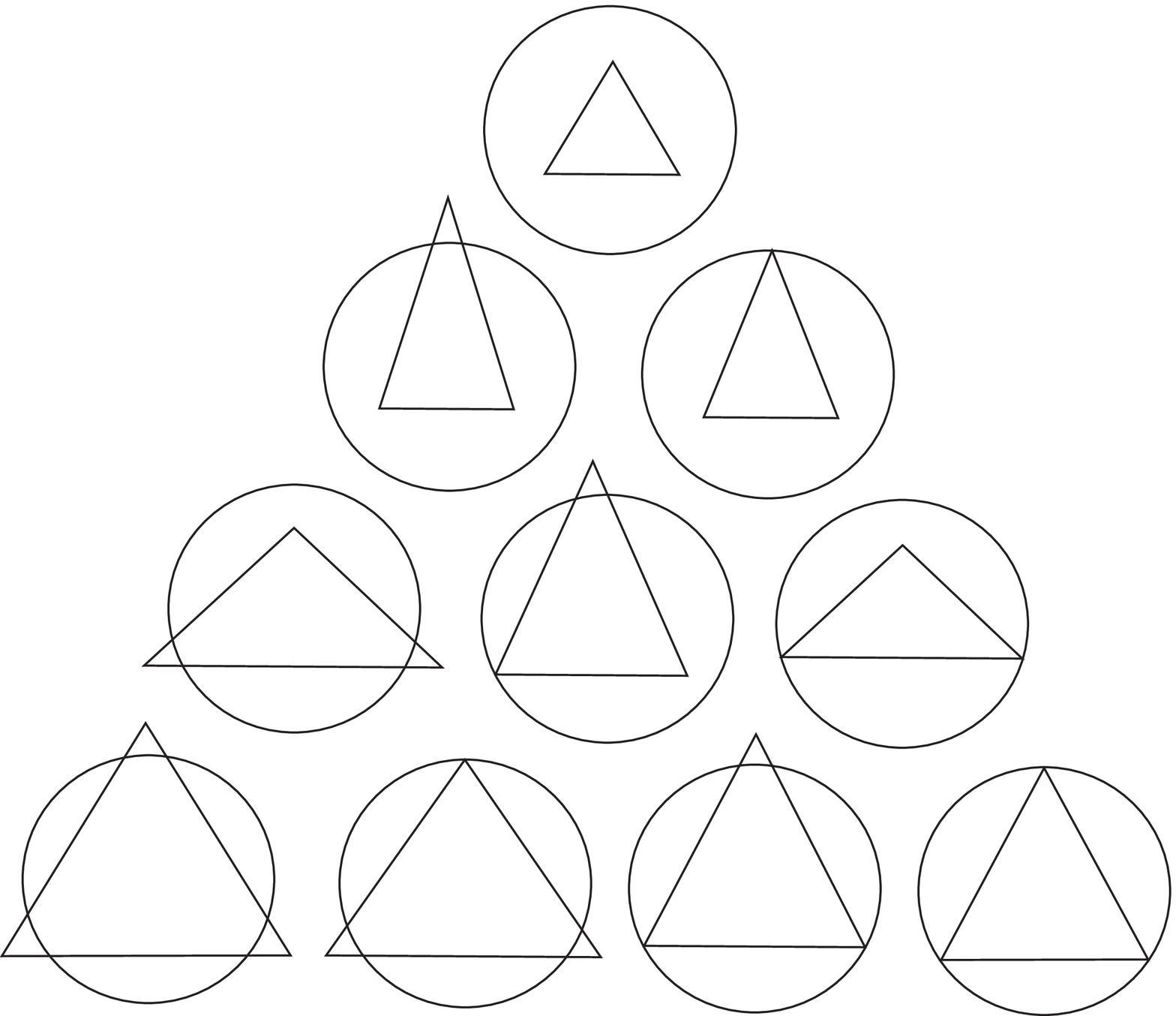}
\end{center}
\caption{\label{fig:triangles-in-klein-model}}

\begin{center}
\includegraphics[scale=.45]{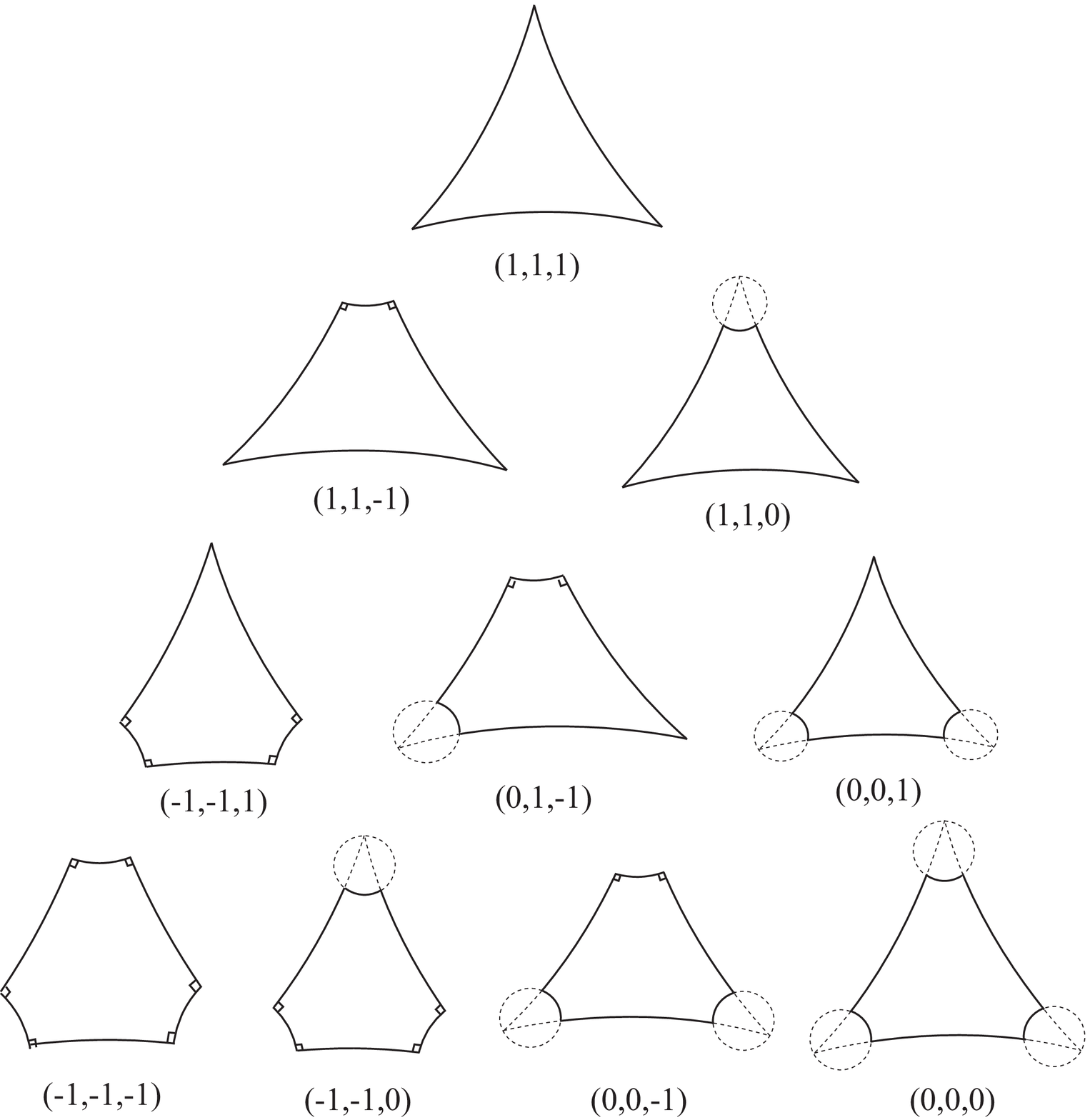}
\end{center}
\caption{\label{fig:ten-triangles}}

We observe that there is a uniform way to write the cosine laws in
all these cases (Lemma \ref{thm:cosine-law}). Furthermore, there
is a uniform formula for the derivative cosine laws (Lemma
\ref{thm:derivative-cosine}). From the derivative cosine laws, we
are able to find the complete list of localized energy functionals
as in \cite{l4}. These action functionals provide variational
principles for cellular decomposed surfaces. All rigidity results
obtained in this paper and the work of Thurston \cite{th}, Penner
\cite{p}, Bobenko-Springborn \cite{bs} can be deduced from those
concave energy functionals.

\subsection{Generalized hyperbolic triangles} A \it decorated convex polygon \rm in the hyperbolic
plane $\mathbb{H}^2$ is a finite area convex polygon $P$ so that
each ideal vertex of $P$ is associated with a horodisk centered at
the vertex. A \it generalized hyperbolic triangle \rm (or simply a
generalized triangle) $\triangle$ in $\mathbb{H}^2$ is a decorated
convex polygon in $\mathbb{H}^2$ bounded by three distinct
geodesics $L_1,L_2,L_3$ and all other (if any) geodesics $L_{ij}$
perpendicular to both $L_i$ and $L_j$. The complete list of all of
them are in Figure \ref{fig:ten-triangles}. We call $L_i\cap
\triangle$ an \it edge \rm of $\triangle.$ In the Klein model of
$\mathbb{H}^2$, there exists a Euclidean triangle
$\widetilde{\triangle}$ in $\mathbb{R}^2$ so that each edge of
$\widetilde{\triangle}$ corresponds to $L_1,L_2$ or $L_3.$

The vertices of $\widetilde{\triangle}$ are called \it
(generalized) vertices \rm of $\triangle$. Note that if $v$ is a
vertex of $\widetilde{\triangle}$ outside $\mathbb{H}^2\cup
\partial\mathbb{H}^2$, then $v$ corresponds to the geodesic $L_{ij}$
perpendicular to the two edges $L_i$ and $L_j$ adjacent to the
vertex $v$. The \it generalized angle \rm (or simply angle) $a(v)$
at a vertex $v$ of $\triangle$ is defined as follows. Let $L_i,
L_j$ be the edges adjacent to $v.$ If $v\in \mathbb{H}^2,$ then
$a(v)$ is the inner angle of $\triangle$ at $v$; if $v\in
\partial\mathbb{H}^2,$ $a(v)$ is TWICE of the length of
the intersection of the associated horocycle with the cusp bounded
by $L_i$ and $L_j$; if $v\notin\mathbb{H}^2\cup
\partial\mathbb{H}^2$, then $a(v)$ is the distance between $L_i$ and
$L_j$. Note that a generalized angle is always positive.

\begin{center}
\includegraphics[scale=.5]{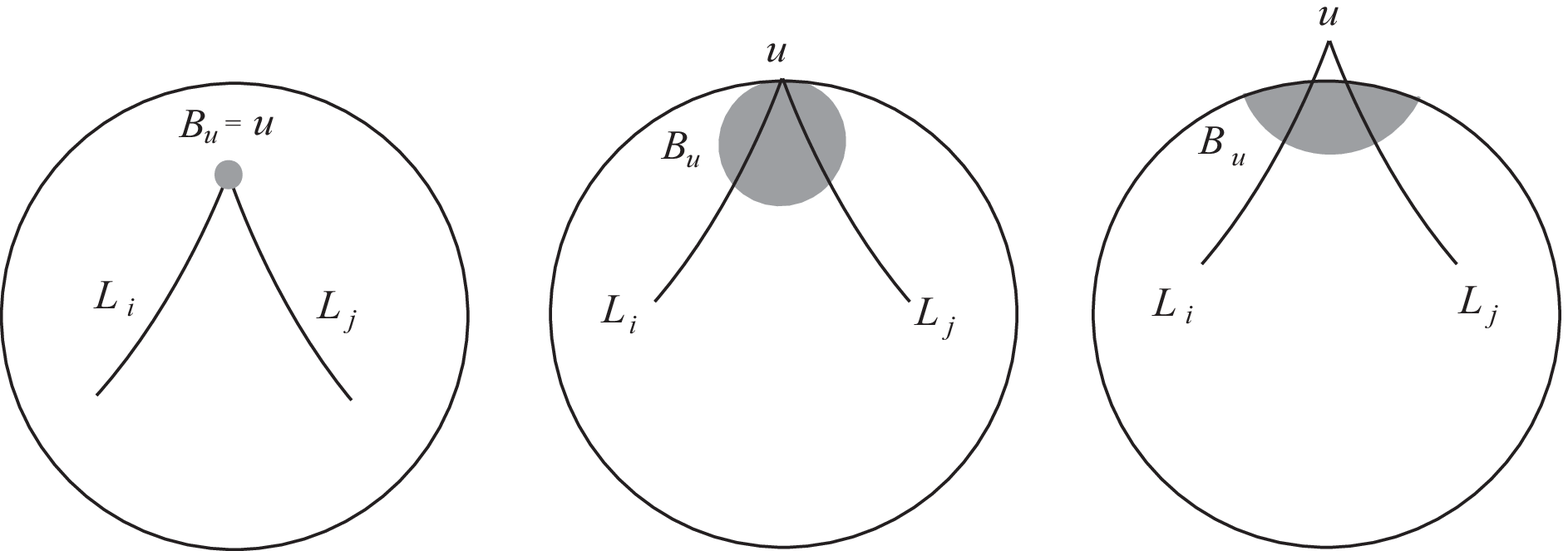}
\end{center}
\caption{\label{fig:generalized-vertex}}

As in Figure \ref{fig:generalized-vertex}, for a generalized
vertex $u$ of $\triangle$, let $B_u=\{u\}$ if $u\in\mathbb{H}^2;$
$B_u$ be the horodisc at $u$ if $u\in\partial\mathbb{H}^2;$ and
$B_u$ is the half plane missing $\triangle$ bounded by $L_{ij}$ if
$u\notin\mathbb{H}^2\cup
\partial\mathbb{H}^2$. The \it generalized edge length \rm (or edge length for
simplicity) of $L_i$ is defined as follows. The generalized length
of the edge $L_i\cap\triangle$ with vertices $u,v$ is the distance
from $B_u$ to $B_v$ if $B_u\cap B_v=\emptyset$ and is the negative
of the distance from $\partial B_u\cap L_i$ to $\partial B_v\cap
L_i$ if $B_u\cap B_v\neq\emptyset$. Note that generalized edge
length may be a negative number. See Figure \ref{fig:distance}.

\begin{center}
\includegraphics[scale=.35]{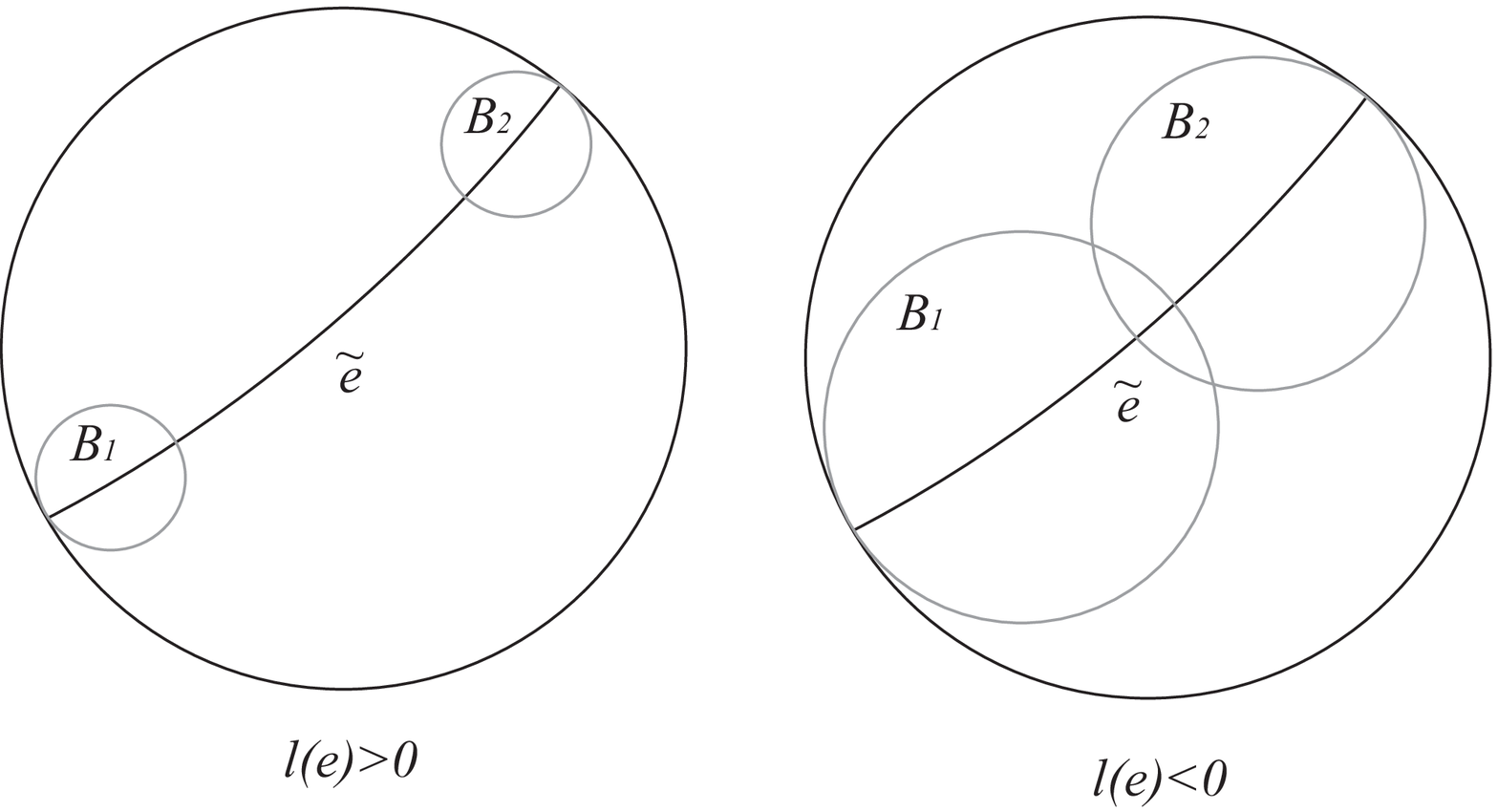}
\end{center}
\caption{\label{fig:distance}}

The cosine law (Lemma \ref{thm:cosine-law}) for generalized
triangles relates the generalized angles with the generalized edge
lengths. Given a generalized triangle $\triangle$ and a vertex $v$
of $\triangle$, the type $\varepsilon$ of $v$ is defined to be
$\varepsilon=1$ if $v\in \mathbb{H}^2,$ $\varepsilon=0$ if $v\in
\partial\mathbb{H}^2$ and $\varepsilon=-1$ if $v\notin\mathbb{H}^2\cup
\partial\mathbb{H}^2$. In this way, generalized triangles are
classified into ten types
$(\varepsilon_1,\varepsilon_2,\varepsilon_3)$ where
$\varepsilon_i\in\{-1,0,1\}$ as in Figure \ref{fig:ten-triangles}.

\subsection{The work of Penner and its generalization}

Suppose $(\widetilde{S},\widetilde{T})$ is a triangulated closed
surface $\widetilde{S}$ with the set of vertices $V$, the set of
edges $E$. We call $T=\{\sigma-V|$ a simplex
$\sigma\in\widetilde{T}\}$ an \it ideal triangulation \rm of the
punctured surface $S=\widetilde{S}-V.$ We call $V$ ideal vertices
(or cusps) of the surface $S.$ If the Euler characteristic of $S$
is negative, a \it decorated hyperbolic metric \rm $(d,r)$ on $S,$
introduced by Penner \cite{p}, is a complete hyperbolic metric $d$
of finite area on $S$ so that each ideal vertex $v$ is assigned a
positive number $r_v$. Let $T_c(S)$ be the Teichm\"uller space of
complete hyperbolic metrics with cusps ends on $S.$ Then the
decorated Teichm\"uller space introduced in \cite{p} is
$T_c(S)\times \mathbb{R}^V_{>0}.$

\begin{center}
\includegraphics[scale=.45]{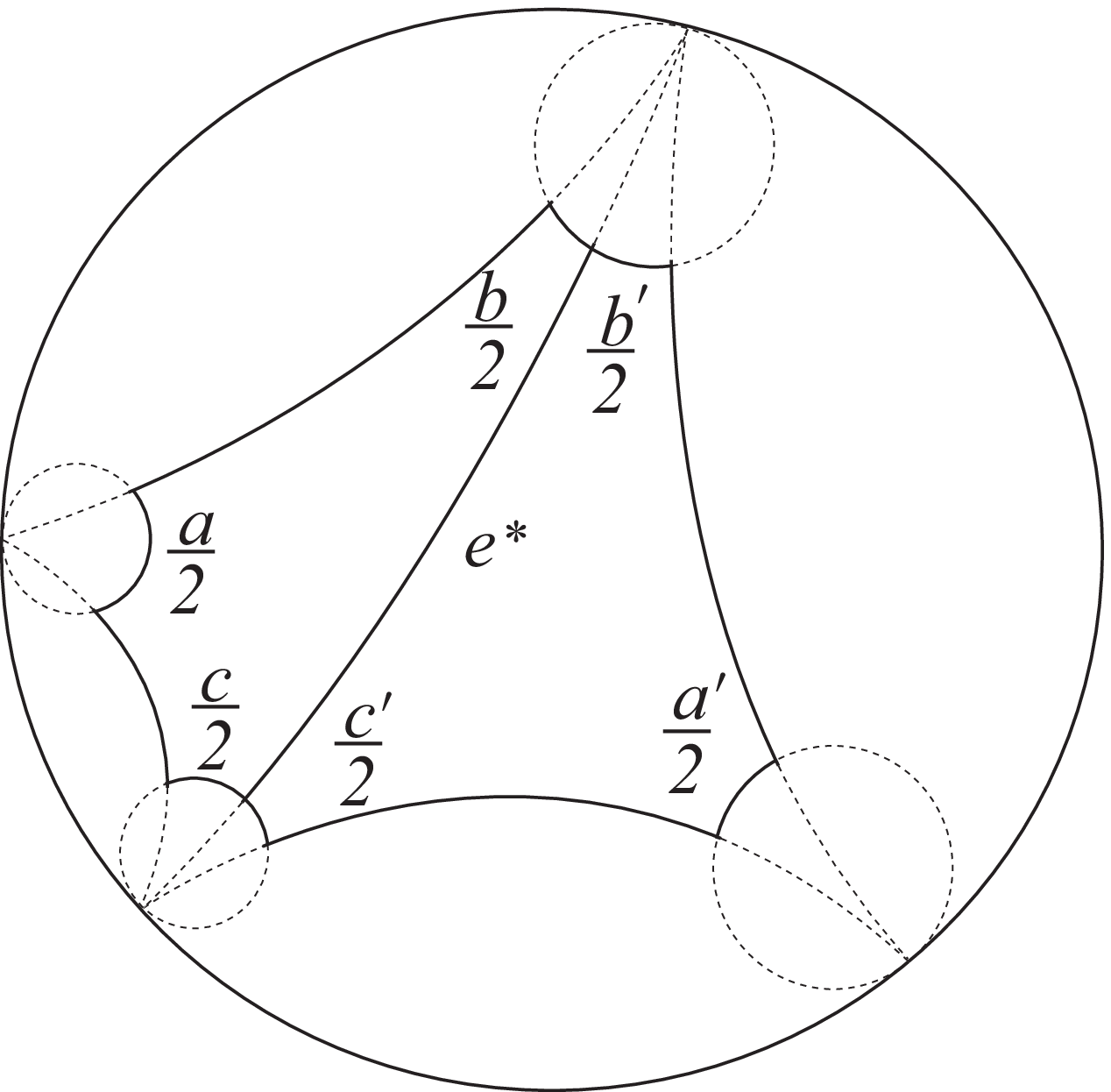}
\end{center}
\caption{\label{fig:simplicial-coordinate}}

Given a decorated hyperbolic metric $(d,r)\in T_c(S)\times
\mathbb{R}^V_{>0},$ using the ideal triangulation $T$, Penner
defined a map $\Psi: T_c(S) \times \mathbb{R}^V_{>0} \to
\mathbb{R}^E$  as follows. Given a metric $(d,r)$, each edge $e\in
E$ is isotopic to a complete geodesic $e^*$ and each triangle
$\sigma$ in $T$ is isotopic to an ideal triangle $\sigma^*$ in the
metric $d$. Since assigning a positive number $r_v$ to each cusp
$v$ is the same as associating a horodisk $B$ centered at the cusp
so that the length of $\partial B$ is $r_v,$ we see that each
ideal triangle $\sigma^*$ is naturally a decorated ideal triangle,
i.e., a type $(0,0,0)$ generalized hyperbolic triangle. The value
of $\Psi(d,r)$ at an edge $e\in E,$ is
$$\Psi(d,r)(e)=\frac{b+c-a}2+\frac{b'+c'-a'}2,$$ where $a,a'$ are the generalized
angle facing $e^*$, $b,b',c,c'$ are the generalized angle adjacent
to $e^*$ as labelled in Figure \ref{fig:simplicial-coordinate}.

An \it edge cycle \rm $(e_1,t_1,e_2,t_2,...,e_k,t_k)$ in a
triangulation $T$ is an alternating sequence of edges $e_i$'s and
faces $t_i$'s in $T$ so that adjacent faces $t_i$ and $t_{i+1}$
share the same edge $e_i$ for any $i=1,...,k$ and $t_{k+1}=t_1.$

A beautiful theorem proved by Penner is the following.

\begin{theorem}[Penner \cite{p}] Suppose $(S,T)$ is an ideally triangulated
surface of negative Euler characteristic. Then for any vector
$z\in\mathbb{R}^E_{\geq0}$ so that $\sum_{i=1}^k z(e_i)
>0$ for any edge cycle $(e_1,t_1,...,e_k,t_k)$, there exists a
unique decorated complete hyperbolic metric $(d,r)$ on $S$  so
that $\Psi(d,r) =z$.
\end{theorem}

Using the derivative cosine law associated to the decorated ideal
triangle and the associated energy function, we generalize
Penner's theorem to the following.

\begin{theorem}\label{thm:generalized-Penner} Suppose $(S,T)$ is an ideally triangulated
surface of negative Euler characteristic. Then Penner's map
$\Psi:T_c(S)\times \mathbb{R}^V_{>0}\to \mathbb{R}^E$ is a smooth
embedding whose image is the convex polytope $P(T)=\{z \in
\mathbb{R}^E|\sum_{i=1}^k z(e_i) >0$ for any edge cycle
$(e_1,t_1,...,e_k,t_k)$\}.
\end{theorem}

We remark that, from the definition, set $P(T)$ is convex. It is
in fact a convex polytope defined by the finite set of linear
inequalities $\sum_{i=1}^k z(e_i) >0$ for those edge cycles
$(e_1,t_1,...,e_k,t_k)$ where each edge appears at most twice (see
\cite{l3}, \cite{g}).

Results similar to Penner's work have been established recently
for hyperbolic cone metric by Leibon \cite{le} and hyperbolic
metric with geodesic boundary in \cite{l3}. In \cite{l4}, a
one-parameter family of coordinates depending on a parameter $h\in
\mathbb{R}$ is introduced for hyperbolic cone metric and
hyperbolic metric with geodesic boundary. These coordinates
generalized the ones in \cite{le} and \cite{l3}. It can be shown
that Penner's map $\Psi$  cannot be deformed. The relationship
between the edge invariant in \cite{l3} and Penner's map $\Psi$
was established by Mondello \cite{m} recently.

\subsection{Thurston-Andreev's circle packing and its generalizations}

\begin{center}
\includegraphics[scale=.5]{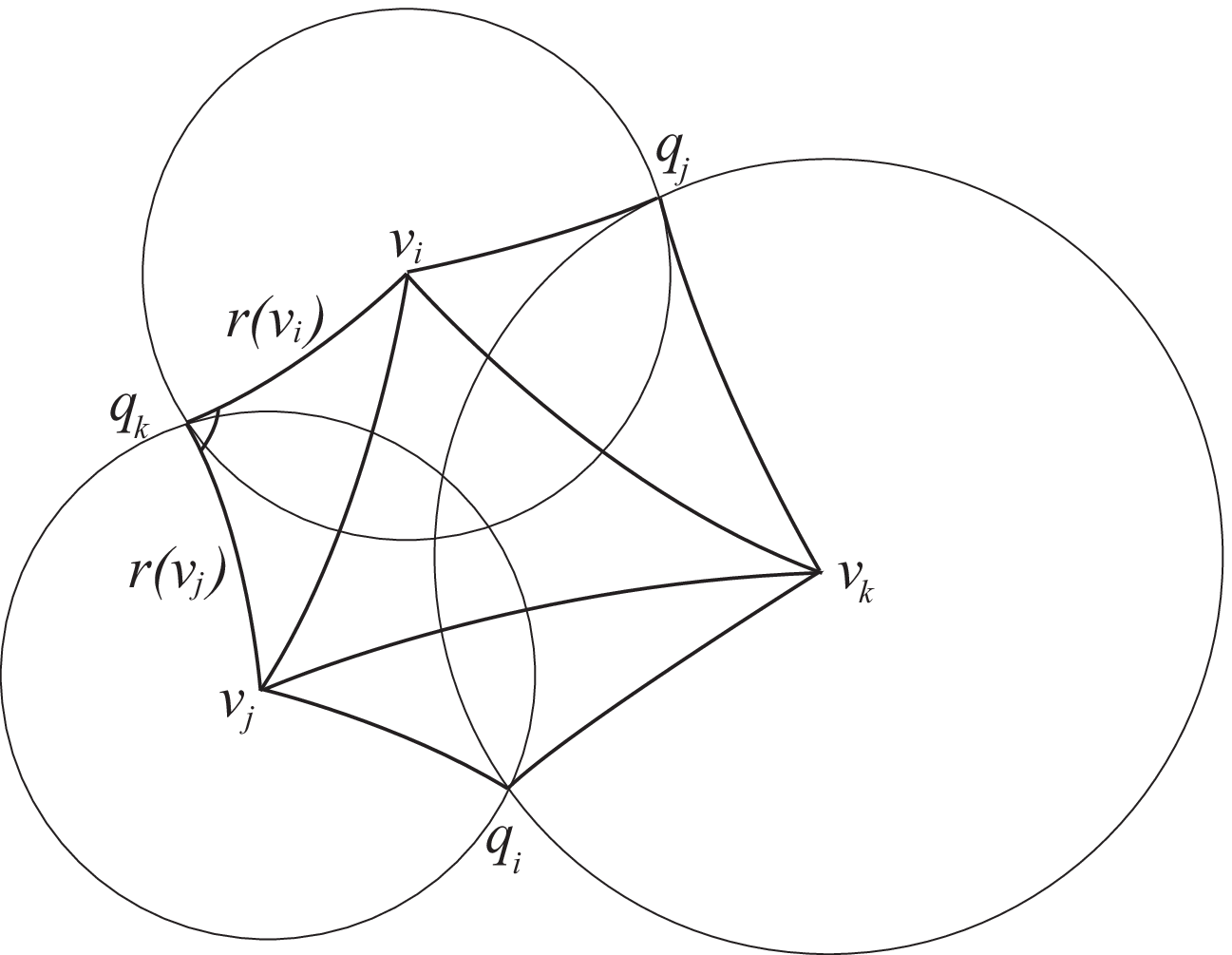}
\end{center}
\caption{\label{fig:Thurston-packing}}

Thurston's work on circle packing can be summarized as follows.
Suppose $(\Sigma,T)$ is a triangulated closed surface so that
$V,E,F$ are sets of all vertices,edges and triangles in $T.$ Fix a
map $\Phi:E\to [\frac\pi2,\pi].$ According to Thurston \cite{th},
a hyperbolic circle packing metric with intersection angles $\Phi$
is a function $r:V\to\mathbb{R}_{>0}$ so that the associated edge
length function $l:E\to \mathbb{R}_{>0}$ is defined as follows. In
Figure \ref{fig:Thurston-packing}, consider a topological triangle
with vertices $v_i,v_j,v_k$. One can construct a hyperbolic
triangle $\triangle v_iv_jq_k$ such that the edges $v_iq_k,v_jq_k$
have lengths $r(v_i),r(v_j)$ respectively and the angle at $q_k$
is $\Phi(v_iv_j)$. Let $l(v_iv_j)$ be the length of edge $v_iv_j$
in the hyperbolic triangle $\triangle v_iv_jq_k$ which is a
function of $r(v_i),r(v_j).$ Similarly, one obtains the edge
lengthes $l(v_jv_k),l(v_kv_i).$

Under the assumption that $\Phi:E\to [\frac\pi2,\pi]$, Thurston
observed that lengths $l(v_iv_j),l(v_jv_k)$ and $l(v_kv_i)$
satisfy the triangle inequality for each triangle $\triangle
v_iv_jv_k$ in $F$. Thus there exists a hyperbolic polyhedral
metric on $(\Sigma,T)$ whose edge length function is $l.$ Let
$K:V\to \mathbb{R}$ be the discrete curvature of the polyhedral
metric, which sends a vertex to $2\pi$ less the sum of all inner
angles at the vertex.

\begin{theorem}[Thurston \cite{th}] \label{thm:circle-packing}For any closed triangulated surface
$(\Sigma,T)$ and $\Phi:E\to [\frac\pi2,\pi],$ a hyperbolic circle
packing metric on $(\Sigma,T)$ is determined by its discrete
curvature, i.e., the map from $r$ to $K$ is injective.
Furthermore, the set of all $K$'s is an open convex polytope in
$\mathbb{R}^V.$
\end{theorem}

\begin{center}
\includegraphics[scale=.6]{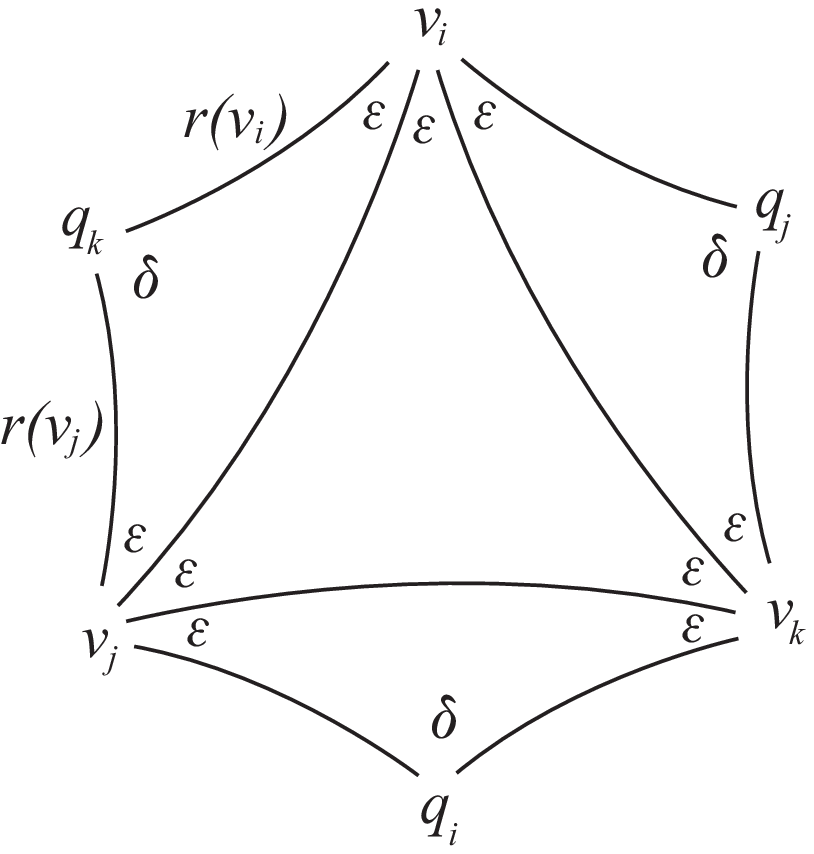}
\end{center}
\caption{\label{fig:generalized-packing}}

Since there are many other cosine laws available, we may try to
use these other cosine laws for generalized triangles of type
$(\varepsilon,\varepsilon,\delta)$ instead of type $(1,1,1)$ used
by Thurston. To state our result, let us fix the notation once and
for all. Let
\begin{equation}\label{fml:I_delta}
I_\delta=\left\{
\begin{array}{lll}
\mathbb{R}_{>0}&\ \ \mbox{if}\ \ \delta=0,-1, \\
(0,\pi]&\ \ \mbox{if}\ \ \delta=1.
\end{array}
\right. \ \ \  \mathring{I}_\delta=\left\{
\begin{array}{lll}
\mathbb{R}_{>0}&\ \ \mbox{if}\ \ \delta=0,-1, \\
(0,\pi)&\ \ \mbox{if}\ \ \delta=1.
\end{array}
\right.
\end{equation}
\begin{equation}\label{fml:J_sigma}
J_\sigma=\left\{
\begin{array}{lll}
\mathbb{R}_{>0}&\ \ \mbox{if}\ \ \sigma=1,-1, \\
\mathbb{R}&\ \ \mbox{if}\ \ \sigma=0.
\end{array}
\right.
\end{equation}

A \it generalized circle packing metric of type
$(\varepsilon,\varepsilon,\delta)$ \rm on a triangulated surface
$(\Sigma,T)$ with weight $\Phi:E\to I_\delta$ is given by a radius
function $r:V\to J_{\varepsilon\delta}$ so that the edge length
function $l:E\to J_{\varepsilon}$ is obtained from the radius $r$
and weight $\Phi$ by the cosine law applied to the generalized
triangle of type $(\varepsilon,\varepsilon,\delta)$. In Figure
\ref{fig:generalized-packing}, consider a triangle with vertices
$v_i,v_j,v_k$ in the triangulation. One can construct a
generalized hyperbolic triangle $\triangle v_iv_jq_k$ of type
$(\varepsilon,\varepsilon,\delta)$ such that the edges
$v_iq_k,v_jq_k$ have lengths $r(v_i),r(v_j)$ respectively and the
generalized angle at $q_k$ is $\Phi(v_iv_j)$. Let $l(v_iv_j)$ be
the length of edge $v_iv_j$ in the generalized hyperbolic triangle
$\triangle v_iv_jq_k$ which is a function of $r(v_i),r(v_j).$
Similarly, one obtains the edge lengthes $l(v_jv_k),l(v_kv_i).$

Depending on $\varepsilon \in \{0,-1,1\}$, the numbers
$l(v_iv_j),l(v_jv_k),l(v_kv_i)$ may not be the three edge lengthes
of a type $(\varepsilon,\varepsilon,\varepsilon)$ triangle. To
this end, let
$\mathcal{M}_{\varepsilon,\delta}(\Phi(v_iv_j),\Phi(v_jv_k),$
$\Phi(v_kv_i))=\{(r(v_i),r(v_j),r(v_k))\in
J_{\varepsilon\delta}^3|$ there exists a type $(\varepsilon,
\varepsilon, \varepsilon)$ triangle $\Delta v_i v_j v_k$ with edge
lengths $l(v_iv_j),l(v_jv_k),l(v_kv_i)$\}. Therefore $r$ can only
take values in $\mathcal{N}_{\varepsilon,\delta}(\Phi)$ a subspace
of $(J_{\varepsilon\delta})^V,$ where
$\mathcal{N}_{\varepsilon,\delta}(\Phi)=\{r:V\to
J_{\varepsilon\delta}|(r(v_i),r(v_j),r(v_k))\in\mathcal{M}_{\varepsilon,\delta}(\Phi(v_iv_j),\Phi(v_jv_k),\Phi(v_kv_i)),$
if $v_i,v_j,v_k$ are vertices of a triangle\}.

The edge length function $l:E\to J_{\varepsilon}$ produces a
polyhedral metric on $(\Sigma,T).$ We define \it the generalized
discrete curvature \rm of the polyhedral metric to be
$\widetilde{K}:V\to \mathbb{R}_{>0}$ sending a vertex to the sum
of all generalized angles at the vertex.

We show that Thurston's circle packing theorem can be generalized
to the following six cases corresponding to the generalized
triangles of type $(-1,-1,1)$, $(-1,-1,-1)$, $(-1,-1,0)$,
$(0,0,1)$, $(0,0,-1)$, $(0,0,0)$ in Figure
\ref{fig:symmetric-triangles-circle-packing}. More precisely, we
list the six cases below.
\begin{enumerate}
\item[(a)] For the case of $(-1,-1,1),$ where $\Phi:E\to (0,\pi]$,
the edge length $l(v_iv_j)$ is obtained from the radii
$r(v_i),r(v_j)$ by the cosine law for the hyperbolic pentagon as
Figure \ref{fig:symmetric-triangles-circle-packing} (a).

\item[(b)] For the case of $(-1,-1,-1),$ where $\Phi:E\to
(0,\infty),$ the edge length $l(v_iv_j)$ is obtained from the
radii $r(v_i),r(v_j)$ by the cosine law for the right-angled
hexagon as Figure \ref{fig:symmetric-triangles-circle-packing}
(b).

\item[(c)] For the case of $(-1,-1,0)$, where $\Phi:E\to
(0,\infty),$ the edge length $l(v_iv_j)$ is obtained from the
radii $r(v_i),r(v_j)$ by the cosine law for the hexagon as Figure
\ref{fig:symmetric-triangles-circle-packing} (c).

\item[(d)] For the case of $(0,0,1),$ where $\Phi:E\to (0,\pi]$,
the edge length $l(v_iv_j)$ is obtained from the radii
$r(v_i),r(v_j)$ by the cosine law for the pentagon as Figure
\ref{fig:symmetric-triangles-circle-packing} (d).

\item[(e)] For the case of $(0,0,-1),$ where $\Phi:E\to
(0,\infty),$ the edge length $l(v_iv_j)$ is obtained from the
radii $r(v_i),r(v_j)$ by the cosine law for the hexagon as Figure
\ref{fig:symmetric-triangles-circle-packing} (e).

\item[(f)] For the case of $(0,0,0)$, where $\Phi:E\to
(0,\infty),$ the edge length $l(v_iv_j)$ is obtained from the
radii $r(v_i),r(v_j)$ by the cosine law for the hexagon as Figure
\ref{fig:symmetric-triangles-circle-packing} (f).
\end{enumerate}

\begin{center}
\includegraphics[scale=.45]{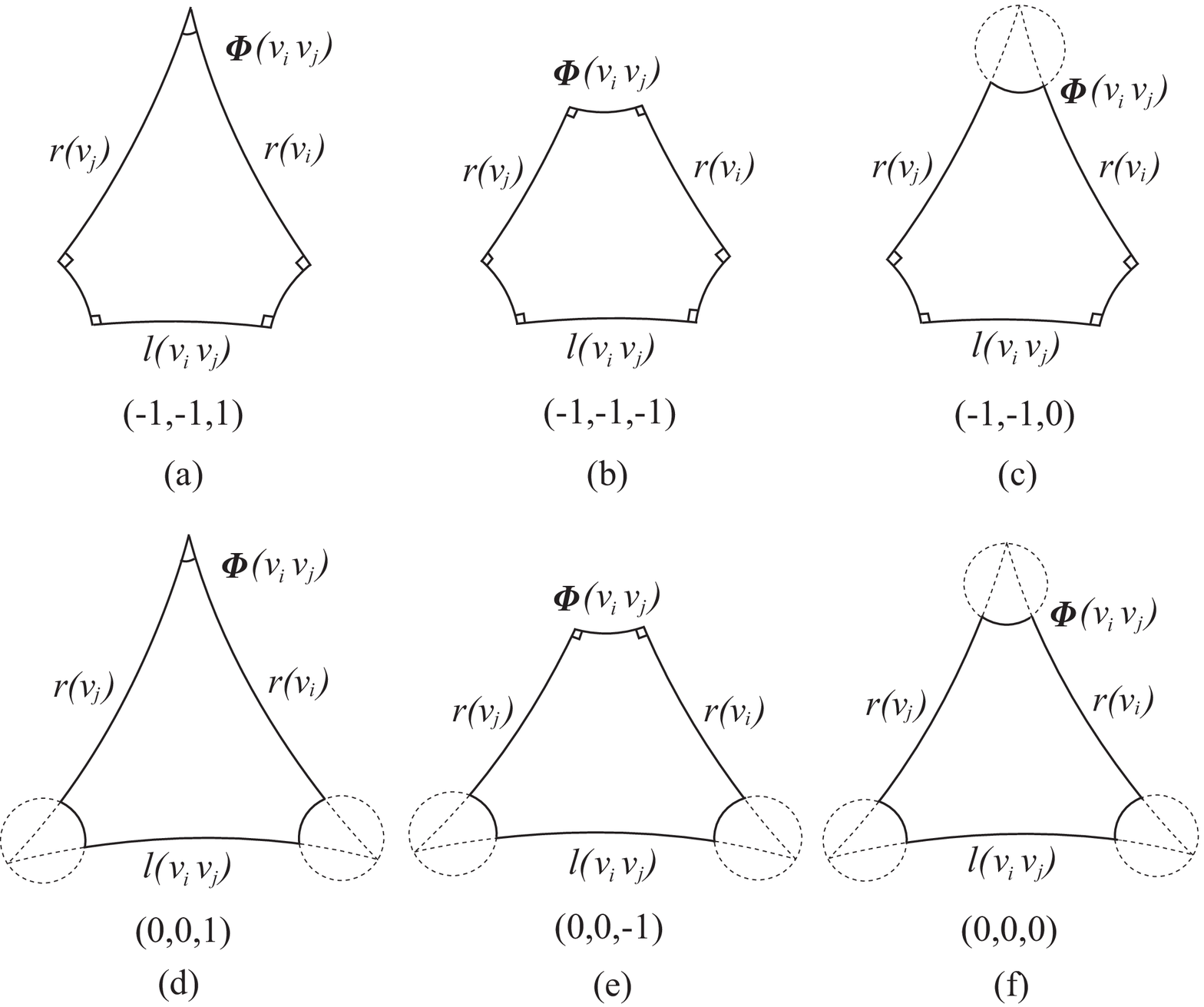}
\end{center}
\caption{\label{fig:symmetric-triangles-circle-packing}}

\begin{theorem}\label{thm:generalized-circle-packing}
Given a closed triangulated surface $(\Sigma,T)$ and $\Phi:E\to
I_\delta$ in the above six cases, the generalized $(\epsilon,
\epsilon, \delta)$ type circle packing metric
$r\in\mathcal{N}_{\varepsilon,\delta}(\Phi)$ is determined by its
generalized discrete curvature $\widetilde{K}:V\to
\mathbb{R}_{>0},$ i.e., the map from $r$ to $\widetilde{K}$ is a
smooth embedding. Furthermore, the set of all $\widetilde{K}$'s is
the space $\mathbb{R}^V_{>0}.$
\end{theorem}

Our method of proof of Theorem
\ref{thm:generalized-circle-packing} also produces a new
variational proof of the rigidity of circle packing in Thurston's
theorem (Theorem \ref{thm:circle-packing}) similar to the proof in
\cite{cl}. However, unlike the proof in \cite{cl} which uses
Thurston's geometric argument and Maple program, our proof is a
straight forward calculation. We are not able to establish Theorem
\ref{thm:generalized-circle-packing} for the rest two cases of
$(1,1,-1),(1,1,0)$ in Figure \ref{fig:two-not-work}.

\begin{center}
\includegraphics[scale=.45]{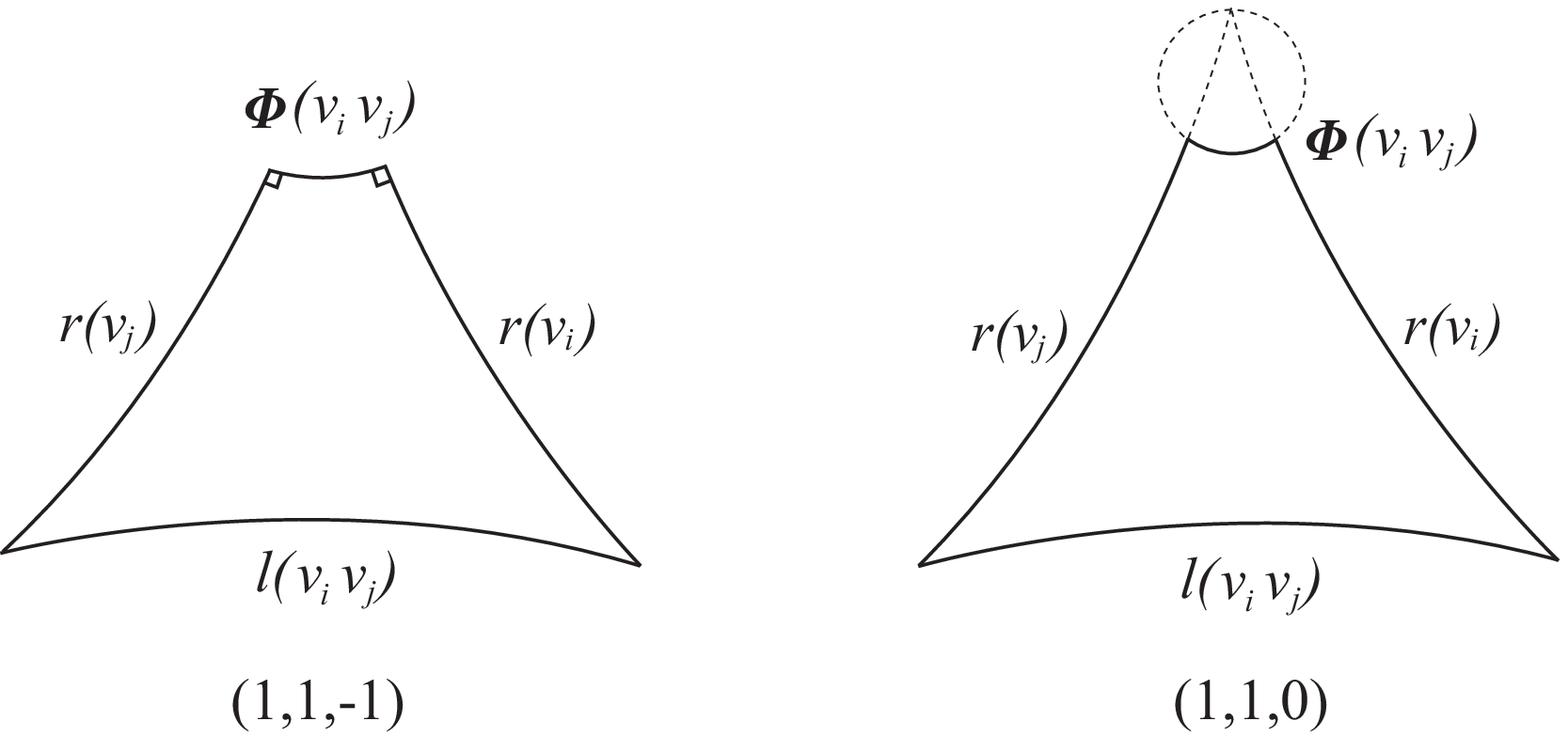}
\end{center}
\caption{\label{fig:two-not-work}}

Generalized circle patterns have been considered by many authors
including Bobenko-Springborn \cite{bs}, Schlenker \cite{sc},
Stephenson \cite{s} and others. In particular, Stephenson's
question \cite{s} (page 331) about disjoint circle pattern
motivates us to consider the above situations.

Furthermore, similar to the work of \cite{cl} on discrete
curvature flow, for the case of $\varepsilon=-1$ in Theorem
\ref{thm:generalized-circle-packing}, there exists a corresponding
generalized curvature flow. Indeed, let $r_i:=r(v_i)$ be the radii
at vertex $v_i\in V$ and $\widetilde{K}_i$ be the generalized
discrete curvature at vertex $v_i\in V.$ The generalized curvature
flow is
$$\frac{dr_i(t)}{dt}=-\widetilde{K}_i\frac12e^{r_i}-\frac12\varepsilon\delta e^{-r_i}.$$
We have shown that it is a negative gradient flow of a strictly
concave down function after a change of variables.

\subsection{Bobenko-Springborn's circle pattern and its
generalizations}

In \cite{bs}, Bobenko-Springborn generalized Thurston's circle
packing pattern in the case of $\Phi=\pi$ in a different setting.
The energy functional in \cite{bs} was derived from the discrete
integrable system. Let us recall briefly the framework in
\cite{bs}. See Figure \ref{fig:circle-pattern} (a). Let
$(\Sigma,G)$ be a cellular decomposition of a closed surface with
the set of vertices $V,$ edges $E$ and 2-cells $F.$ The dual
cellular decomposition $G^*$ has the set of vertices $V^*(\cong
F)$ so that each 2-cell $f$ in $F$ contains exactly one vertex
$f^*$ in $G^*.$ If $v$ is a vertex of a 2-cell $f$, we denote it
by $v<f.$ For all pairs $(v,f)$ where $v<f$, join $v$ to $f^*$ by
an arc in $f$, denoted by $(v,f^*),$ so that $(v,f^*)$ and
$(v',f^*)$ don't intersect in their interior. Then these arcs
$\cup_{(v,f)}(v,f^*)$ decompose the surface $\Sigma$ into a union
of quadrilaterals of the form $(v,v',f^*,f'^*)$ where $vv'\in E,
v<f,v'<f'.$ According to \cite{bs}, this quadrilateral
decomposition of a closed surface arises naturally from the
integrable system and discrete Riemann surfaces. Now suppose
$\theta: E\to(0,\pi)$ is given. For $r:V^*\to\mathbb{R}_{>0}$,
called a \it circle pattern metric\rm, and a quadrilateral
$(v,v',f^*,f'^*)$, construct an $\mathbb{E}^2$ (or $\mathbb{H}^2$)
triangle $\triangle f^*f'^*v$ so that the length of the edges
$vf^*,vf'^*$ are given by $r(f^*),r(f'^*)$ and the angle at $v$ is
$\theta(vv').$

\begin{center}
\includegraphics[scale=.5]{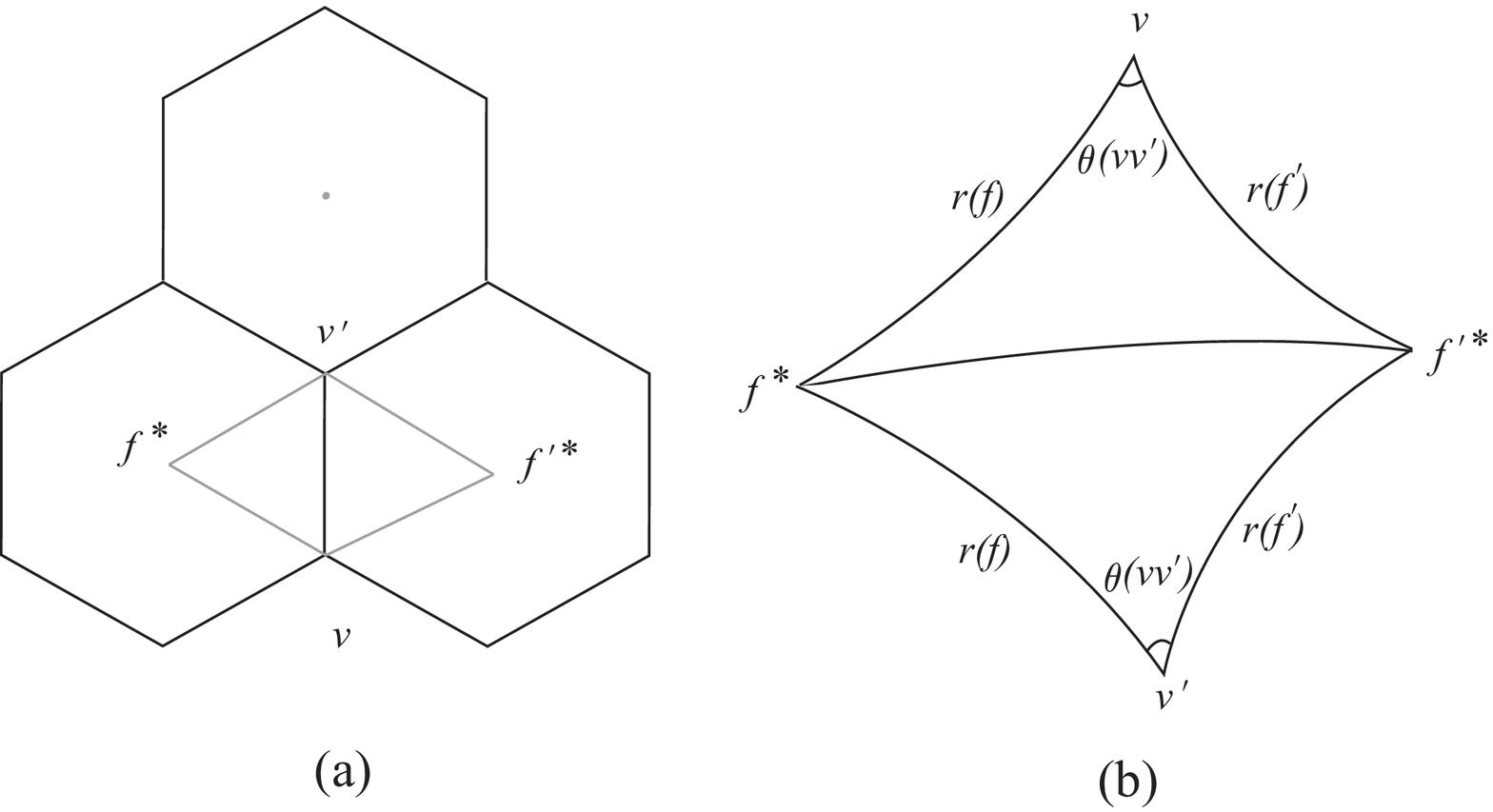}
\end{center}
\caption{\label{fig:circle-pattern}}

In this way the quadrilateral $(v,v',f^*,f'^*)$ is realized in
$\mathbb{E}^2$ (or $\mathbb{H}^2$) as the isometric double of the
triangle $\triangle f^*f'^*v$ across the edge $f^*f'^*$. Since the
surface $\Sigma$ is a gluing of the quadrilateral
$(v,v',f^*,f'^*)$ along edges, by isometrically gluing these
quadrilaterals, one obtains a polyhedral metric on the surface
$\Sigma$ with cone points at the vertices $V$ and $V^*.$ The cone
angles at $v\in V$ are prescribed by $\theta,$ the only variable
curvatures are at $f^*\in V^*.$ Bobenko-Springborn's rigidity
result says,
\begin{theorem}[Bobenko-Springborn \cite{bs}]\label{thm:circle-pattern} For any cell decomposition
$(\Sigma,G)$ of a closed surface and any map $\theta:E\to
(0,\pi),$ the circle pattern metric $r:V^*\to\mathbb{R}_{>0}$ is
determined by its discrete curvature $K:V^*\to\mathbb{R}$ for
hyperbolic polyhedral metrics and is determined up to scaling by
$K:V^*\to\mathbb{R}$ for Euclidean polyhedral metrics.
\end{theorem}

In the hyperbolic geometric setting, the essential part of
Bobenko-Springborn's construction is to produce the hyperbolic
triangle $\triangle f^*f'^*v$ with two prescribed edge lengths and
the prescribed angle between the two edges. Since there are eight
other generalized hyperbolic triangles of type
$(\varepsilon,\varepsilon,\delta)$ as listed in Figure
\ref{fig:ten-triangles}, we can use them to produce hyperbolic
metrics. Our result below shows that the rigidity phenomena still
hold for all other eight cases.

We assume the same setting as in \cite{bs} that $(\Sigma,G)$ is a
cellular decomposed surface so that $V,E,F$ are the sets of all
vertices, edges and 2-cells with dual $(\Sigma,G^*)$. Suppose
$(\varepsilon,\varepsilon,\delta)\in\{-1,0,1\}^3$, and a map
$\theta: E\to \mathring{I}_\delta$ is given (see
(\ref{fml:I_delta}) for the definition of $\mathring{I}_\delta$).

Assume for an $r\in(J_{\varepsilon\delta})^{V^*}$ and any
quadrilateral $(v,v',f^*,f'^*)$, we can construct a type
$(\varepsilon,\varepsilon,\delta)$ generalized hyperbolic triangle
$\triangle f^*f'^*v$ so that the length of $vf^*,vf'^*$ are
$r(f^*),r(f'^*)$ respectively and the generalized angle at $v$ is
$\theta(vv').$ Now realize the quadrilateral $(v,v',f^*,f'^*)$ as
the metric double of $\triangle f^*f'^*v$ across the edge
$f^*f'^*.$ The \it generalized curvature \rm of the resulting
circle pattern is concentrated at the vertices $V^*.$ It is
defined as follows. For $h\in\mathbb{R},$ define the generalized
curvature $K_h:V^*\to\mathbb{R}$ by
\begin{equation}\label{fml:K_h}
K_h(f^*)=\sum_{i=1}^m2\int_1^{a_i}\rho_\varepsilon^h(t)dt
\end{equation}
where $a_i'$s are the generalized angles at the vertex $f^*$ in
the triangle $\triangle f^*f'^*v$ and
$\rho_\varepsilon(t)=\int_0^t\cos(\sqrt{\varepsilon}x)dx$ (see the
definition in \S 2).

\begin{theorem}\label{thm:generalized-circle-pattern} Under the assumption above, for any
$(\varepsilon,\varepsilon,\delta)\in\{-1,0,1\}^3$ and $\theta:E\to
\mathring{I}_\delta$, the map from $(J_{\varepsilon\delta})^{V^*}$
to $\mathbb{R}^{V^*}$ sending $r$ to $K_h$ is a smooth embedding.
\end{theorem}

Bobenko-Springborn's circle pattern theorem (Theorem
\ref{thm:circle-pattern}) in the hyperbolic geometry corresponds
to $(\varepsilon,\varepsilon,\delta)=(1,1,1)$ and $h=0$ in Theorem
\ref{thm:generalized-circle-pattern}. Bobenko-Springborn \cite{bs}
also showed that the image of $\{K\}$ is an explicit open convex
polytope in $\mathbb{R}^{V^*}$. It is an interesting question to
investigate the images of $\{K_h\}$ in the generalized setting.

We remark that there is a discrete curvature flow for each case in
Theorem \ref{thm:generalized-circle-pattern}. Let $r_i:=r(f^*_i).$
Then
$$\frac{dr_i(t)}{dt}=-K_h(f^*_i)(\frac12e^{r_i}-\frac12\varepsilon\delta
e^{-r_i})^{1-h}.$$ We have shown that the above equation is the
negative gradient flow of a strictly concave down function.
Similar situation have been considered before by Hazel. In
\cite{ha}, he considered the flow for cases when $h=0$ and
$(\varepsilon,\delta,\varepsilon)=(1,1,1)$, $(1,1,-1)$ or
$(1,1,0).$

\section{A proof of Theorem \ref{thm:generalized-Penner}}

We give a proof of Theorem \ref{thm:generalized-Penner} in this
section. The proof consists of two parts. In the first part we
show that Penner's map $\Psi$ is an embedding. Then we determine
its image. As in \cite{l3}, the rigidity follows from the
following well-known fact together with the cosine law for
decorated ideal triangle.

\begin{lemma}\label{thm:convex} If $X$ is an open convex set in $\mathbb{R}^n$
and $f:X\to\mathbb{R}$ is smooth strictly convex, then the
gradient $\nabla f:X\to\mathbb{R}^n$ is injective. Furthermore, if
the Hessian of $f$ is positive definite for all $x \in X$, then
$\nabla f$ is a smooth embedding.
\end{lemma}

To begin, recall that $(S,T)$ is an ideally triangulated surface
with sets of edges, triangles and cusps given by $E,F,V.$ We
assume that $\chi(S)<0.$

Following Penner \cite{p}, we will produce a smooth
parametrization of the decorated Teichm\"uller space
$T_c(S)\times\mathbb{R}^V_{>0}$ by $\mathbb{R}^E$ using the edge
lengths. From the derivative cosine law for decorated ideal
triangle, we will construct a smooth strictly concave down
function $H$ on $\mathbb{R}^E$ so that its gradient is Penner's
map $\Psi$. Then by Lemma \ref{thm:convex}, the map
$\Psi:T_c(S)\times\mathbb{R}^V_{>0}\to\mathbb{R}^E$ is an
embedding. To determine the image
$\Psi(T_c(S)\times\mathbb{R}^V_{>0}),$ we study the degenerations
of decorated ideal triangles. The strategy of the proof is the
same as that in \cite{l3}.

\subsection{Penner's length parametrization of
$T_c(S)\times\mathbb{R}^V_{>0}$}\label{3.1} For each decorated
hyperbolic metric $(d,r)\in T_c(S)\times\mathbb{R}^V_{>0},$ where
$r=(r_1,...r_{|V|})$, one replaces each edge $e\in E$ by the
geodesic $e^*$ in the metric $d$ and constructs for each cusp
$v_i\in V$ a horocyclic disk $B_{r_i}(v_i)$ centered at $v_i$
whose circumference $\partial B_{r_i}(v_i)$ has length
$r_i=r(v_i).$ Now, the \it length coordinate \rm
$l_{d,r}\in\mathbb{R}^E$  of $(d,r)$ is defined as follows. Given
$r: V \to \mathbb{R}_{>0}$, realize each triangle $\Delta u v w$
in $T$ by a decorated ideal hyperbolic triangle with generalized
angles at $u,v,w$ being $r(u), r(v), r(w)$. Then $l_{d,r}(e)$ is
the generalized edge length of the edge $e = uv$ in the triangle
$\Delta uvw$.
 In this way, Penner
defined a length map
\begin{align*}
L:T_c(S)\times\mathbb{R}^V_{>0}&\to\mathbb{R}^E\\
(d,r)&\to l_{d,r}.
\end{align*}

\begin{lemma} [Penner \cite{p}] The length map
$L:T_c(S)\times\mathbb{R}^V_{>0}\to\mathbb{R}^E$ is a
diffeomorphism.
\end{lemma}

\begin{proof} By the cosine law for decorated ideal triangle, the
map $L$ satisfies for all $\lambda\in\mathbb{R}_{>0}$
\begin{align}\label{fml:length-map}
L(d,\lambda r)=L(d,r)-(2\ln\lambda)(1,1,...,1).
\end{align} Thus, it suffices to deal with those decorated metrics
$(d,r)$ so that $r$ are small, i.e., $L(d,r)\in\mathbb{R}^E_{>0}.$
In this case, Penner proved that
$L|:L^{-1}(\mathbb{R}^E_{>0})\to\mathbb{R}^E_{>0}$ is a
diffeomorphism by a direct geometric construction using isometric
gluing decorated ideal triangles. By (\ref{fml:length-map}), it
follows that $L$ is a diffeomorphism.
\end{proof}

\subsection{Penner's map $\Psi$ is a coordinate}

Recall that for a decorated ideal triangle $\triangle$ with edges
$e_1,e_2,e_3$ of lengths $l_1,l_2,l_3$ and opposite generalized
angles $\theta_1,\theta_2,\theta_3$, the cosine law obtained by
Penner \cite{p}
 says:
$$\frac{e^{l_i}}2=\frac2{\theta_j\theta_k},\   \ \frac{\theta_i^2}4=e^{l_i-l_j-l_k}.$$
where $\{i,j,k\}=\{1,2,3\}.$ The derivative cosine law expressing
$l_i$ in terms of $(\theta_1, \theta_2, \theta_3)$ says
$$ \frac{\partial l_i}{\partial \theta_i} =0,   \  \  \
\frac{\partial l_i}{\partial \theta_j} = -\frac{1}{\theta_j}.$$

 Let $x_i=\frac12(\theta_j+\theta_k-\theta_i)$ (or
$\theta_i=x_j+x_k$). We call $x_i$ the \it radius invariant \rm at
the edge $e_i$ in the triangle $\triangle$.

Using the derivative cosine law, we have,
\begin{lemma} \label{thm:closed-form-ideal} Under the same assumption as above, the differential
1-form $\omega=\sum_{i=1}^3x_idl_i$ is closed in $\mathbb{R}^3$ so
that its integration $W(l)=\int_0^l\omega$ is strictly concave
down in $\mathbb{R}^3$. Furthermore,
\begin{align}\label{fml:W}
\frac{\partial W}{\partial l_i}=x_i
\end{align}
\end{lemma}
\begin{proof} Consider the matrix
 $H=[\partial l_{a}/\partial x_b]_{3 \times 3}$. The closeness of $\omega$ is equivalent to
 that $H$ is symmetric.  The strictly concavity of $W$ will be a consequence of the positive definiteness of $H$.
We establish these two properties for $H$ as follows.
 Assume that
indices $\{i,j,k\}=\{1,2,3\}$. By definition,
$\frac{\partial}{\partial x_i}=\frac{\partial}{\partial \theta_j}
+\frac{\partial}{\partial \theta_k}$. It follows from the
derivative cosine law that,
$$\frac{\partial l_i}{\partial x_i} = -(\frac{1}{\theta_j}
+\frac{1}{\theta_k}),$$ and
$$ \frac{\partial l_i}{\partial x_j} = -\frac{1}{\theta_k}$$
which is symmetric in $i,j$. This shows that the  matrix $H$ is
symmetric. Furthermore, the negative matrix $-H$ is of the form
$[m_{ab}]_{3 \times 3}$ where $ m_{ab} =m_{ba} >0$ and
$m_{ii}=m_{ij}+m_{ik}$.  The determinant of such a matrix can be
calculated easily as $4 m_{11}m_{22} m_{33} >0$, and the
determinant of a principal $2 \times 2$ submatrix is
$m_{ii}m_{jj}-m_{ij}^2 = ( m_{ij}+m_{ik}) (m_{ij}+m_{jk})-m_{ij}^2
>0$. It follows that $-H$ is positive definite.
\end{proof}

By the construction in $\S$ \ref{3.1}, it suffices to show that
the composition $\widetilde{\Psi}=\Psi\circ
L^{-1}:\mathbb{R}^E\to\mathbb{R}^E$ is a smooth embedding. The map
$\widetilde{\Psi}$ is constructed explicitly as follows. For each
$l\in\mathbb{R}^E$ and each triangle $\sigma\in F$ realize
$\sigma$ by an ideal hyperbolic triangle together with horocycles
centered at three vertices so that the generalized edge length of
an edge $e$ in $\sigma$ is $l(e).$ Now isometrically glue these
ideal hyperbolic triangles along edges so that the horocycles
match. The result is a complete finite area hyperbolic metric on
the surface $S$ together with a horocycle at each cusp. For each
edge $e\in E,$ the value $\widetilde{\Psi}(l)(e)$ is equal to
$\frac{b+c-a}2+\frac{b'+c'-a'}2$ where $a,a',b,b',c,c'$ are
generalized angles  facing and adjacent to $e$ in Figure
\ref{fig:simplicial-coordinate}. Thus
\begin{align}\label{fml:F}
\widetilde{\Psi}(l)(e)=r_f(e)+r_{f'}(e)
\end{align}
where $f,f'$ are the decorated ideal triangles sharing the edge
$e$ and $r_f(e),r_{f'}(e)$ are the radius invariants at the edge
$e$ in $f,f'$ respectively.

Given a vector $l\in\mathbb{R}^E,$ define the energy $H(l)$ of $l$
to be $$H(l)=\sum_{\{i,j,k\}\in F}W(l(e_i),l(e_j),l(e_k))$$ where
the sum is over all triangles $\{i,j,k\}$ in $F$ with edges
$e_i,e_j,e_k.$ By definition and Lemma
\ref{thm:closed-form-ideal}, $H:\mathbb{R}^E\to\mathbb{R}$ is
smooth and strictly concave down when Hessian is negative
definite. Furthermore, by (\ref{fml:W}) and (\ref{fml:F}),
$$\frac{\partial H}{\partial l(e_i)}=\widetilde{\Psi}(l)(e_i)$$ i.e., $\nabla H=\widetilde{\Psi}.$
It follows from Lemma \ref{thm:convex}, that
$\widetilde{\Psi}:\mathbb{R}^E\to\mathbb{R}^E$ is a smooth
embedding. Therefore $\Psi$ is a smooth embedding.

\begin{center}
\includegraphics[scale=.35]{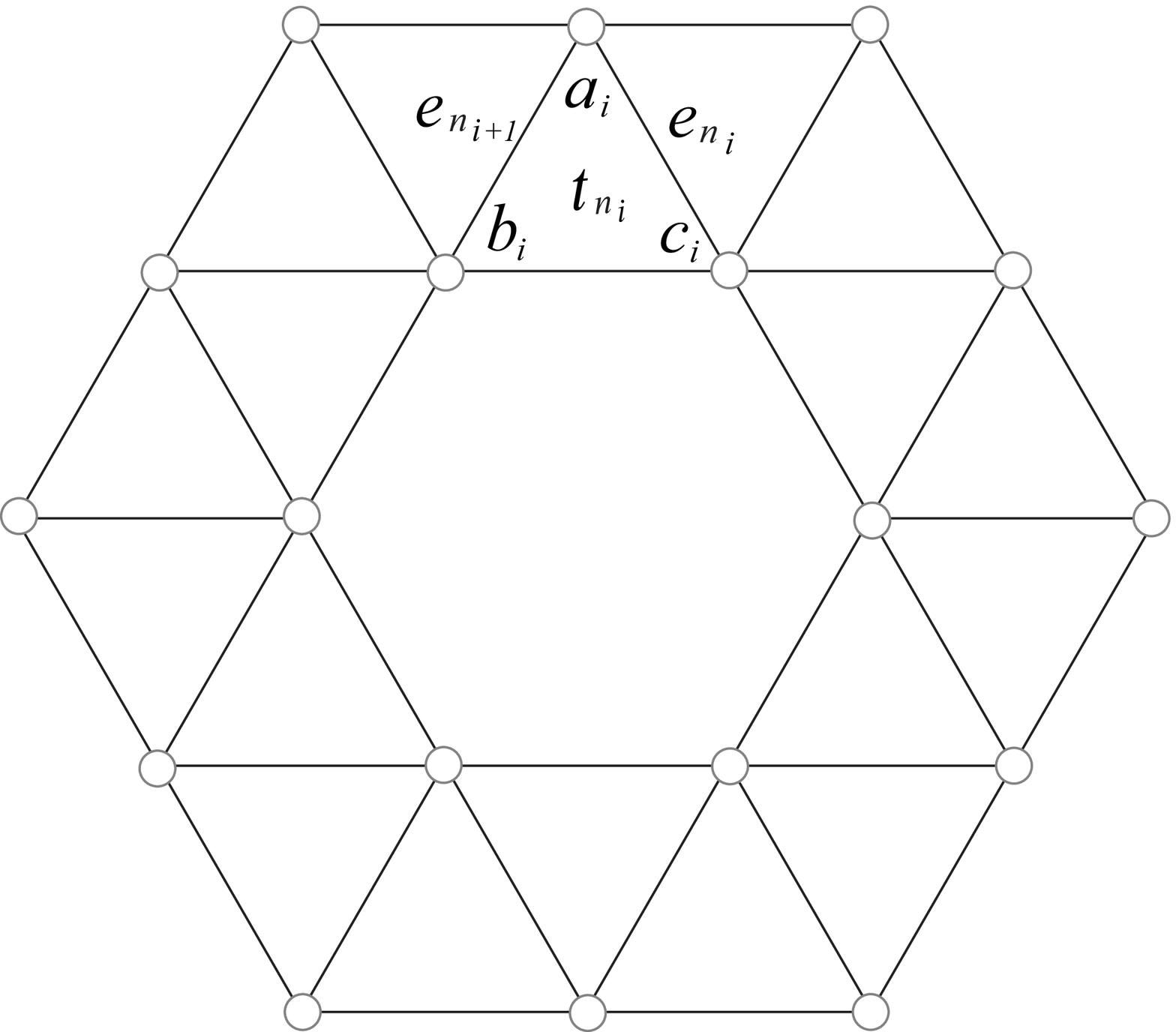}
\end{center}
\caption{\label{fig:cycle}}

\subsection{The image of Penner's map} Let $\Omega$ be the convex
subset of $\mathbb{R}^E$ given by $\Omega=\{z\in
\mathbb{R}^E|\sum_{i=1}^pz(e_{n_i})>0,$ whenever
$(e_{n_1},t_{n_1},e_{n_2},t_{n_2},...,e_{n_p},t_{n_p},e_{n_1})$ is
an edge cycle\}. To show that
$\widetilde{\Psi}(\mathbb{R}^E)=\Omega,$ due to convexity of
$\Omega,$ it suffices to prove that
$\widetilde{\Psi}(\mathbb{R}^E)$ is both open and closed in
$\Omega.$

First to see that $\widetilde{\Psi}(\mathbb{R}^E)\subset\Omega,$ take
$l\in\mathbb{R}^E$ and an edge cycle
$(e_{n_1},t_{n_1},e_{n_2},t_{n_2},$ $...,e_{n_p},t_{n_p},e_{n_1})$
as shown in Figure \ref{fig:cycle}.

Let the generalized angles in the decorated ideal triangle
$t_{n_i}$ in the metric $L^{-1}(l)$ be $a_i,b_i,c_i,$ where $b_i$
faces the edge $e_{n_i},$ $c_i$ faces the edge $e_{n_{i+1}}$ and
$a_i$ is adjacent to $e_{n_i},e_{n_{i+1}}.$ Then the contribution
to $\sum_{j=1}^p\widetilde{\Psi}(l)(e_{n_j})$ from
$e_{n_i},e_{n_{i+1}}$ in triangle $t_{n_i}$ is given by
$\frac{a_i+b_i-c_i}2+\frac{a_i+c_i-b_i}2=a_i.$ Thus
\begin{align}\label{fml:sum}
\sum_{j=1}^p\widetilde{\Psi}(l)(e_{n_j})=\sum_{j=1}^pa_i>0
\end{align}
due to $a_i>0$ for all $i.$ It follows that
$\widetilde{\Psi}(\mathbb{R}^E)$ is open in $\Omega$ since
$\widetilde{\Psi}:\mathbb{R}^E\to\mathbb{R}^E$ is just proved to
be an embedding.

It remains to prove that $\widetilde{\Psi}(\mathbb{R}^E)$ is
closed in $\Omega.$ The closeness of
$\widetilde{\Psi}(\mathbb{R}^E)$ in $\Omega$ means to show that if
a sequence $\{l^{(m)}\in\mathbb{R}^E\}_{m=0}^{\infty}$ satisfies
$\lim_{m \to \infty}\widetilde{\Psi}(l^{(m)})=z\in\Omega,$ then
$\{l^{(m)}\}_{m=0}^\infty$ contains a subsequence converging to a
point in $\mathbb{R}^E.$ Given a decorated hyperbolic metric
$l^{(m)}\in\mathbb{R}^E$ on $(S,T)$ and a generalized angle
$\theta$, let $\theta^{(m)}\in\mathbb{R}^{3F}$ be the generalized
angles of the decorated ideal triangles in $(S,T)$ in the metric
$l^{(m)}$. By taking a subsequence if necessary, we may assume
that $\lim_{m \to \infty}\widetilde{\Psi}(l^{(m)})$ converges in
$[-\infty, \infty]^E$ and that for each angle $\theta_i,$ the
limit $\lim_{m \to \infty}\theta_i^{(m)}$ exists in $[0, \infty]$.

\begin{lemma}\label{thm:infinity} For all $i,$ $\lim_{m \to
\infty}\theta_i^{(m)}\in[0, \infty).$
\end{lemma}
\begin{proof} If otherwise, say $\lim_{m \to
\infty}\theta_1^{(m)}=\infty$ for some angle $\theta_1$. Let
$e_1,e_2$ be the edges adjacent to the angle $\theta_1$ in the
triangle $t$. Take an edge cycle
$(e_{n_1},t_{n_1},e_{n_2},t_{n_2},...,$ $e_{n_p},t_{n_p},e_{n_1})$
which contains $(e_1,t,e_2)$ as a part. Then by the calculation as
in (\ref{fml:sum}), $$\sum_{i=1}^pz(e_{n_i})=\lim_{m \to
\infty}\sum_{i=1}^p\widetilde{\Psi}(l^{(m)})(e_{n_i})\geq\lim_{m
\to \infty}\theta_i^{(m)}=\infty.$$ This contradicts the
assumption that $z\in\Omega.$
\end{proof}

Now, we finish the proof by contradiction as follows. If $\lim_{m
\to \infty}l^{(m)}$ were not in $\mathbb{R}^E,$ there would exist
an edge $e\in E$ so that $\lim_{m \to
\infty}l^{(m)}(e)=\pm\infty.$ Let $t$ be a triangle adjacent to
$e.$ Let $\theta_1^{(m)},\theta_2^{(m)}$ be the generalized angles
in $t$ adjacent to $e$ in the metric $l^{(m)}.$ By the cosine law,
\begin{align}\label{fml:exp}
\exp(l^{(m)}(e))=\frac4{\theta_1^{(m)}\theta_2^{(m)}}
\end{align}
and $\theta_1^{(m)},\theta_2^{(m)}\in(0,\infty).$

Case 1. If $\lim_{m \to \infty}l^{(m)}(e)=-\infty,$ then $\lim_{m
\to \infty}\exp(l^{(m)}(e))=0.$ By (\ref{fml:exp}), one of $\lim_{m
\to \infty}\theta_i^{(m)}$ must be $\infty.$ But this contradicts
Lemma \ref{thm:infinity}.

Case 2. If $\lim_{m \to \infty}l^{(m)}(e)=\infty,$ then $\lim_{m \to
\infty}\exp(l^{(m)}(e))=\infty.$ Since Lemma \ref{thm:infinity}
shows that $\theta_1^{(m)}$ and $\theta_2^{(m)}$ are bounded, by
(\ref{fml:exp}), one of the limits $\lim_{m \to
\infty}\theta_i^{(m)}$ must be zero for $i=1$ or $2.$ Say $\lim_{m
\to \infty}\theta_1^{(m)}=0.$ Let $e_1$ be the other edge in the
triangle $t$ so that $e_1, e$ are adjacent to the generalized angle
$\theta_1$. Let $\theta_3$ be the third angle in $t,$ facing $e.$
Then by the cosine law that
$\exp(l^{(m)}(e_1))=\frac4{\theta_1^{(m)}\theta_3^{(m)}}$ and Lemma
\ref{thm:infinity} on the boundedness of $\theta_3^{(m)},$ we
conclude that $\lim_{m \to \infty}l^{(m)}(e_1)=\infty.$ To
summarize, from $\lim_{m \to \infty}l^{(m)}(e)=\infty$ and any
triangle $t$ adjacent to $e$, we conclude that there is an edge
$e_1$ and an angle $\theta_1$ adjacent to $e, e_1$ in $t$ so that
$\lim_{m \to \infty}l^{(m)}(e_1)=\infty$ and $\lim_{m \to
\infty}\theta_1^{(m)}=0.$

Applying this procedure to $e_1$ and the triangle $t_1$ adjacent
to $e_1$ from the side other than the side that $t$ lies. We
obtain the next angle, say $\theta_2$ and edge $e_2$ in $t_1$ so
that $\lim_{m \to \infty}l^{(m)}(e_2)=\infty$ and $\lim_{m \to
\infty}\theta_2^{(m)}=0.$ Since there are only finite number of
edges and triangles, this procedure will produce an edge cycle
$(e_{n_1},t_{n_1},e_{n_2},t_{n_2},...,e_{n_k},t_{n_k},e_{n_1})$ in
$T$ so that

(i) $\lim_{m \to \infty}l^{(m)}(e_{n_i})=\infty$ for each $i$,

(ii) $\lim_{m \to \infty}\theta_i^{(m)}=0$ for each $i$, where
$\theta_i$ is the angle in triangle $t_{n_i}$ adjacent to
$e_{n_i}$ and $e_{n_{i+1}}.$

By (\ref{fml:sum}),  $$\sum_{i=1}^kz(e_{n_i})=\lim_{m \to
\infty}\sum_{i=1}^k\widetilde{\Psi}(l^{(m)})(e_{n_i})=\lim_{m \to
\infty}\sum_{i=1}^k\theta_i^{(m)}=0.$$ This contradicts the
assumption that $z\in\Omega.$

\section{The derivative cosine law}

We give a unified approach to all cosine laws and sine laws in
this section. The derivative of the cosine laws are also
determined. Most of the proof are straightforward checking and
will be delayed to the appendix.

Assume that a generalized hyperbolic triangle of type
$(\varepsilon_1,\varepsilon_2,\varepsilon_3)\in\{-1,0,1\}^3$ has
generalized angles $\theta_1,\theta_2,\theta_3$ and opposite
generalized edge lengthes $l_1,l_2,l_3.$ There are ten different
types of generalized hyperbolic triangles as shown in Figure
\ref{fig:ten-triangles}. The relationships between $l_i'$s and
$\theta_i'$s are expressed in the cosine law and the sine law. To
state them, we introduce the two functions $\rho_{\varepsilon},
\tau_s$ depending on $\varepsilon,\theta,s,l\in \mathbb{R}$:
\begin{align}
\rho_\varepsilon(\theta)&=\int_0^\theta\cos(\sqrt{\varepsilon} x)dx
=\frac{1}{\sqrt{\varepsilon}}\sin(\sqrt{\varepsilon} x),\\
\tau_s(l)&=\frac12e^l-\frac12se^{-l}.
\end{align}
To be more explicitly,
$$\begin{array}{lll}
\rho_1(\theta)=\sin(\theta),&\rho_0(\theta)=\theta,&\rho_{-1}(\theta)=\sinh(\theta,\\
\tau_1(l)=\sinh(l),&\tau_0(l)=\frac12e^l,&\tau_{-1}(l)=\cosh(l).
\end{array}$$

A simple calculation shows,
\begin{align*}
\rho'_\varepsilon(\theta)&:=\frac{\partial
\rho_\varepsilon(\theta)}{\partial \theta}
=\cos(\sqrt{\varepsilon} \theta),\\
\tau'_s(l)&:=\frac{\partial \tau_s(l)}{\partial
l}=\frac12e^l+\frac12se^{-l}.
\end{align*}

\begin{lemma}[The cosine laws and the sine laws]\label{thm:cosine-law} For a generalized hyperbolic triangle
of type $(\varepsilon_1,\varepsilon_2,\varepsilon_3)\in\{-1,0,1\}^3$
with generalized angles $\theta_1,\theta_2,\theta_3$ and opposite
generalized edge lengths $l_1,l_2,l_3$, for $\{i,j,k\}=\{1,2,3\}$,
the following hold:
\begin{align}
\tau'_{\varepsilon_j\varepsilon_k}(l_i)&=
\frac{\rho'_{\varepsilon_i}(\theta_i)+\rho'_{\varepsilon_j}(\theta_j)\rho'_{\varepsilon_k}(\theta_k)}
{\rho_{\varepsilon_j}(\theta_j)\rho_{\varepsilon_k}(\theta_k)},\label{fml:cosine1}\\
2\rho_{\varepsilon_i}^2(\frac{\theta_i}2)&=
\frac{\tau'_{\varepsilon_j\varepsilon_k}(l_i)-\frac12\varepsilon_je^{l_j-l_k}-\frac12\varepsilon_ke^{l_k-l_j}}
{\tau_{\varepsilon_k\varepsilon_i}(l_j)\tau_{\varepsilon_i\varepsilon_j}(l_k)},\label{fml:cosine2}\\
\rho'_{\varepsilon_i}(\theta_i)&=
\frac{-\varepsilon_i\tau'_{\varepsilon_j\varepsilon_k}(l_i)+\tau'_{\varepsilon_k\varepsilon_i}(l_j)\tau'_{\varepsilon_i\varepsilon_j}(l_k)}
{\tau_{\varepsilon_k\varepsilon_i}(l_j)\tau_{\varepsilon_i\varepsilon_j}(l_k)},\label{fml:cosine3}\\
\frac{\rho_{\varepsilon_i}(\theta_i)}{\tau_{\varepsilon_j\varepsilon_k}(l_i)}&=
\frac{\rho_{\varepsilon_j}(\theta_j)}{\tau_{\varepsilon_k\varepsilon_i}(l_j)}\label{fml:sine}
\end{align}
\end{lemma}
\begin{proof} In the Appendix A, the cosine law and the sine law for
each type of the ten generalized hyperbolic triangles are listed.
We check directly that all the formulas there fit the uniform
formulas (\ref{fml:cosine1})-(\ref{fml:sine}).
\end{proof}

\begin{remark} The identity (\ref{fml:sine}) is called the \it sine law. \rm
The formula (\ref{fml:cosine2}) is stronger than the formula
(\ref{fml:cosine3}). They are equivalent if $\varepsilon_i=\pm1$.
If $\varepsilon_i=0$, the formula (\ref{fml:cosine3}) is trivial
while the formula (\ref{fml:cosine2}) expresses $\theta$ in term
of lengths $l$.

For the six cases of generalized triangle without ideal vertices,
there is a unified strategy to derive the cosine law
(\ref{fml:cosine1}) and (\ref{fml:cosine3}) by using the
hyperboloid model. This unified strategy is essentially given in
\cite{th2} p74-82. But this method does not work for the other
four cases of generalized triangle with ideal vertices.
\end{remark}

The concepts of Gram matrix and angle Gram matrix of a generalized
hyperbolic triangle are defined as follows. For a generalized
hyperbolic triangle of type
$(\varepsilon_1,\varepsilon_2,\varepsilon_3)\in\{-1,0,1\}^3$ with
generalized angles $\theta_1,\theta_2,\theta_3$ and opposite
generalized edge lengths $l_1,l_2,l_3$, its Gram matrix is
$$
G_l:=-\left(
\begin{array}{ccc}
\varepsilon_1&\tau'_{\varepsilon_1\varepsilon_2}(l_3)&\tau'_{\varepsilon_3\varepsilon_1}(l_2) \\
\tau'_{\varepsilon_1\varepsilon_2}(l_3)&\varepsilon_2&\tau'_{\varepsilon_2\varepsilon_3}(l_1) \\
\tau'_{\varepsilon_3\varepsilon_1}(l_2)&\tau'_{\varepsilon_2\varepsilon_3}(l_1)&\varepsilon_3
\end{array}
\right)
$$
and its angle Gram matrix is
$$
G_\theta:=-\left(
\begin{array}{ccc}
-1&\rho'_{\varepsilon_3}(\theta_3)&\rho'_{\varepsilon_2}(\theta_2) \\
\rho'_{\varepsilon_3}(\theta_3)&-1&\rho'_{\varepsilon_1}(\theta_1) \\
\rho'_{\varepsilon_2}(\theta_2)&\rho'_{\varepsilon_1}(\theta_1)&-1
\end{array}
\right).
$$

\begin{lemma}\label{thm:determinant}
\begin{align}
\det
G_l=-(\tau_{\varepsilon_k\varepsilon_i}(l_j)\tau_{\varepsilon_i\varepsilon_j}(l_k)
\rho_{\varepsilon_i}(\theta_i))^2 \label{fml:det G_l}\\
\det
G_\theta=-(\rho_{\varepsilon_j}(\theta_j)\rho_{\varepsilon_k}(\theta_k)\tau_{\varepsilon_j\varepsilon_k}(l_i))^2
\label{fml:det G_a}
\end{align}
\end{lemma}

For a hyperbolic triangle, Lemma \ref{thm:determinant} is
well-known. We will prove it for a generalized triangle in Appendix
B.

By Lemma \ref{thm:determinant} and the sine laws (\ref{fml:sine}),
we see
\begin{align*}
M:=&\frac1{\sqrt{-\det G_l}} \left(
\begin{array}{ccc}
\tau_{\varepsilon_2\varepsilon_3}(l_1)&0&0 \\
0&\tau_{\varepsilon_3\varepsilon_1}(l_2)&0 \\
0&0&\tau_{\varepsilon_1\varepsilon_2}(l_3)
\end{array}
\right)\\
=&\frac1{\sqrt{-\det G_\theta}} \left(
\begin{array}{ccc}
\rho_{\varepsilon_1}(\theta_1)&0&0 \\
0&\rho_{\varepsilon_2}(\theta_2)&0 \\
0&0&\rho_{\varepsilon_3}(\theta_3)
\end{array}
\right).\\
\end{align*}

\begin{lemma}\label{thm:inverse matrix} $MG_lMG_\theta=I.$
\end{lemma}
This lemma is a consequence of the cosine laws and the sine laws
(\ref{fml:cosine1}), (\ref{fml:cosine3}) and (\ref{fml:sine}). It
is checked by direct calculation.

Let $y_1,y_2,y_3$ be three functions of variables $x_1,x_2,x_3$. Let
$A=(\frac{\partial y_i}{\partial x_j})_{3\times3}$ be the Jacobi
matrix. Then the differentials $dy_1,dy_2,dy_3$ and $dx_1,dx_2,dx_3$
satisfy
$$\left(
\begin{array}{ccc}
dy_1 \\
dy_2 \\
dy_3
\end{array}\right)
=A\left(
\begin{array}{ccc}
dx_1 \\
dx_2 \\
dx_3
\end{array}\right).
$$
\begin{lemma}[The derivative cosine law]\label{thm:derivative-cosine} For a generalized hyperbolic triangle
of type $(\varepsilon_1,\varepsilon_2,\varepsilon_3)\in\{-1,0,1\}^3$
with generalized angles $\theta_1,\theta_2,\theta_3$ and opposite
generalized edge lengths $l_1,l_2,l_3$, the differentials of $l'$s
and $\theta'$s satisfy the following relations:
\begin{align}
\left(
\begin{array}{ccc}
dl_1 \\
dl_2 \\
dl_3
\end{array}\right)
=&MG_l \left(
\begin{array}{ccc}
d\theta_1 \\
d\theta_2 \\
d\theta_3
\end{array}\right) \label{fml:derivative1} \\
\left(
\begin{array}{ccc}
d\theta_1 \\
d\theta_2 \\
d\theta_3
\end{array}\right)
=&MG_\theta \left(
\begin{array}{ccc}
dl_1 \\
dl_2 \\
dl_3
\end{array}\right). \label{fml:derivative2}
\end{align}
\end{lemma}

We will prove the above lemma in Appendix B.

In the rest of the section, we establish the existence of a
generalized hyperbolic triangle of type
$(\varepsilon,\varepsilon,\delta)\in\{-1,0,1\}^3$ with two given
edge lengths $l_1,l_2$ and a generalized angle $\theta$ between
them, where the generalized angles opposite to $l_1,l_2$ have type
$\varepsilon$ and $\theta$ has type $\delta.$ Recall
$\mathring{I}_\delta$ and $J_{\varepsilon\delta}$ are in
(\ref{fml:I_delta}) (\ref{fml:J_sigma}).

Fix type $(\varepsilon,\varepsilon,\delta)\in\{-1,0,1\}^3$. For a
given $\theta\in \mathring{I}_\delta,$ let's introduce the set
$\mathcal{D}_{\varepsilon,\delta}(\theta)=\{(l_1,l_2)\in(J_{\varepsilon\delta})^2|$
there exists a generalized hyperbolic triangle of type
$(\varepsilon,\varepsilon,\delta)$ with two edge lengths $l_1,
l_2$ so that the generalized angle between them is $\theta$ of
type $\delta$\}.

\begin{lemma}\label{thm:realize}
\begin{enumerate}
\item If $\varepsilon=1$ or $0,$ then $\mathcal{D}_{\varepsilon,\delta}(\theta)=(J_{\varepsilon\delta})^2.$ \\
\item If $(\varepsilon,\delta)=(-1,1),$ then
$$\mathcal{D}_{\varepsilon,\delta}(\theta)=\{(l_1,l_2)\in\mathbb{R}_{>0}^2|\sinh l_1\sinh
l_2-\cos \theta\cosh l_1\cosh l_2>1\}.$$\\
\item If $(\varepsilon,\delta)=(-1,0),$ then
$$\mathcal{D}_{\varepsilon,\delta}(\theta)=\{(l_1,l_2)\in\mathbb{R}^2|\theta>e^{-l_1}+e^{-l_2}\}.$$\\
\item If $(\varepsilon,\delta)=(-1,-1),$ then
$$\mathcal{D}_{\varepsilon,\delta}(\theta)=\{(l_1,l_2)\in\mathbb{R}_{>0}^2|\cosh \theta\sinh
l_1\sinh l_2-\cosh l_1\cosh l_2>1\}.$$
\end{enumerate}
\end{lemma}

\begin{proof} We will construct a generalized hyperbolic triangle
of type $(\varepsilon,\varepsilon,\delta)$ in each case as follows.

If $(\varepsilon,\varepsilon,\delta)=(1,1,1),$ choose a point $O$
in the hyperbolic plane. Draw two geodesics rays $L_1,L_2$
starting from the point $O$ such that the angle between $L_1,L_2$
is $\theta\in(0,\pi).$ Let $P_i$ be the point on $L_i$ such that
the length of the segment $OP_i$ is $l_i>0, i=1,2.$ Then one
obtains the hyperbolic triangle by joining $P_1,P_2$ by a geodesic
segment.

If $(\varepsilon,\varepsilon,\delta)=(1,1,0),$ choose a point $O$ at
the infinite of the hyperbolic plane and draw a horocycle centered
at $O.$ Let $H_1,H_2$ be two points on the horocycle such that the
length of the horocyclic arc $H_1H_2$ is $\frac\theta2>0.$ For
$i=1,2,$ draw a geodesic $L_i$ passing through $O$ and $H_i.$ Let
$P_i$ be the point on $L_i$ such that the length of the segment of
$H_iP_i$ is $|l_i|$ and $H_i$ is between $O,P_i$ on $L_i$ if and
only if $l_i>0$. Then one obtains the generalized hyperbolic
triangle by joining $P_1,P_2$ by a geodesic segment.

If $(\varepsilon,\varepsilon,\delta)=(1,1,-1),$ draw a geodesic
segment $G_1G_2$ of length $\theta>0$ with end points $G_1,G_2.$
For $i=1,2,$ let $L_i$ be a geodesic ray starting from $G_i$
perpendicular to $G_1G_2$ and $L_1,L_2$ are in the same half plane
bounded by the geodesic containing $G_1G_2.$ Let $P_i$ be the
point on $L_i$ such that the length of the segment $G_iP_i$ is
$l_i>0, i=1,2.$ Then one obtains the generalized hyperbolic
triangle by joining $P_1,P_2$ by a geodesic segment.

If $(\varepsilon,\varepsilon,\delta)=(0,0,1),$ choose a point $O$
in the hyperbolic plane $\mathbb{H}^2$. Draw two geodesics rays
$L_1,L_2$ starting form the point $O$ such that the angle between
$L_1,L_2$ is $\theta\in(0,\pi)$ and $L_i$ ends at point
$P_i\in\partial\mathbb{H}^2.$ Join $P_1,P_2$ by a geodesic. Let
$H_i$ be the point on the geodesic containing $L_i$ such the
length of $OH_i$ is $|l_i|$ and $H_i$ is between $O,P_i$ if and
only if $l_i>0$. Draw a horocycle centered at $P_i$ passing
through $H_i$. One obtains the generalized triangle $OP_1P_2$
decorated by two horocycles.

If $(\varepsilon,\varepsilon,\delta)=(0,0,0)$ or $(0,0,-1)$, the
construction is similar to the case of $(0,0,1)$ above.

If $(\varepsilon,\varepsilon,\delta)=(-1,-1,1),$ choose a point
$O$ in the hyperbolic plane. Draw two geodesics rays $L_1,L_2$
starting form the point $O$ such that the angle between $L_1,L_2$
is $\theta\in(0,\pi).$ For $i=1,2,$ let $P_i$ be the point on
$L_i$ such that the length of the segment $OP_i$ is $l_i>0.$ Let
$M_i$ be the geodesic passing through $P_i$ perpendicular to
$L_i.$ There is a generalized hyperbolic triangle with prescribed
lengths $l_1,l_2$ and angle $\theta$ if and only if the distance
between $M_1,M_2$ is positive. The distance $l$ between $M_1,M_2$
can be calculated from the cosine law of generalized triangle of
type $(-1,-1,1)$:
$$\cosh l=\sinh l_1\sinh l_2-\cos \theta\cosh l_1\cosh l_2.$$ This
shows that the condition in part (2) is equivalent to $l>0.$

If $(\varepsilon,\varepsilon,\delta)=(-1,-1,0),$ choose a point $O$
at the infinite of the hyperbolic plane and draw a horocycle
centered at $O$. Let $H_1,H_2$ be the two points on the horocycle
such that the length of the horocyclic arc $H_1H_2$ is
$\frac\theta2>0.$ For $i=1,2,$ draw a geodesic $L_i$ passing through
$O$ and $H_i.$ Let $P_i$ be the point on $L_i$ such that the length
of the segment $H_iP_i$ is $|l_i|$ and $H_i$ is between $O,P_i$ if
and only if $l_i>0.$ Let $M_i$ be the geodesic passing through $P_i$
perpendicular to $L_i.$ There is a generalized hyperbolic triangle
with prescribed lengths $l_1,l_2$ and generalized angle $\theta$ if
and only if the distance $l$ between $M_1,M_2$ is positive. The
distance $l$ between $M_1,M_2$ can be calculated from the cosine law
of generalized triangle of type $(-1,-1,0)$:
$$\cosh l=\frac12\theta^2e^{l_1+l_2}-\cosh(l_1-l_2).$$
Now, one sees $\frac12\theta^2e^{l_1+l_2}-\cosh(l_1-l_2)>1$ if and
only if $\theta>e^{-l_1}+e^{-l_2}.$ This shows part (3) holds.

If $(\varepsilon,\varepsilon,\delta)=(-1,-1,-1),$ draw a geodesic
segment $G_1G_2$ of length $\theta>0$ with end points $G_1,G_2.$ For
$i=1,2,$ let $L_i$ be a geodesic half line starting from $G_i$
perpendicular to $G_1G_2$ and $L_1,L_2$ are in the same half plane
bounded by the geodesic containing $G_1G_2.$ Let $P_i$ be the point
on $L_i$ such that the length of the segment $G_iP_i$ is $l_i>0,
i=1,2.$ Let $M_i$ be the geodesic passing through $P_i$
perpendicular to $L_i.$ There is a right-angled hexagon with
prescribed lengths $l_1,l_2,\theta$ if and only if the distance $l$
between $M_1,M_2$ is positive. The distance $l$ between $M_1,M_2$
can be calculated from the cosine law of a right-angled hexagon:
$$\cosh l=\cosh \theta\sinh l_1\sinh
l_2-\cosh l_1\cosh l_2.$$ This shows that the condition in part
(4) is equivalent to $l>0.$
\end{proof}

\section{A proof of Theorem \ref{thm:generalized-circle-packing}}

We give a proof of the generalized circle packing theorem (Theorem
\ref{thm:generalized-circle-packing}) in this section.

Recall that $(\Sigma,T)$ is a closed triangulated surface with
$V,E,F$ the sets all vertices, edges and triangles in $T.$ Given a
type $(\varepsilon,\varepsilon,\delta)\in\{-1,0,1\}^3$ and
$\Phi:E\to I_\delta,$ for each
$r\in\mathcal{N}_{\varepsilon,\delta}(\Phi),$ a generalized
$(\varepsilon,\varepsilon,\delta)$ circle packing on $(\Sigma,T)$
is based on the following local construction.

As in Figure \ref{fig:packing-data}, consider a topological
triangle with vertices $v_i,v_j,v_k$. One can construct a
generalized hyperbolic triangle $\triangle v_iv_jq_k$ of type
$(\varepsilon,\varepsilon,\delta)$ such that the edges
$v_iq_k,v_jq_k$ have lengths $r(v_i):=r_i,r(v_j):=r_j$
respectively and the generalized angle at $q_k$ is
$\Phi(v_iv_j)=:\phi_k$. Let $l(v_iv_j)=:l_k$ be the length of edge
$v_iv_j$ in the generalized hyperbolic triangle $\triangle
v_iv_jq_k$ which is a function of $r_i,r_j.$ Similarly, one
obtains the edge lengthes $l(v_jv_k)=:l_i,l(v_kv_i):=l_j.$ Let
$\theta_i$ be the generalized angle of $\triangle v_iv_jv_k$ at
the vertex $v_i.$ With fixed $(\phi_i,\phi_j,\phi_k),$ we consider
$\theta_i$ as a function of $(r_i,r_j,r_k).$

\begin{center}
\includegraphics[scale=.6]{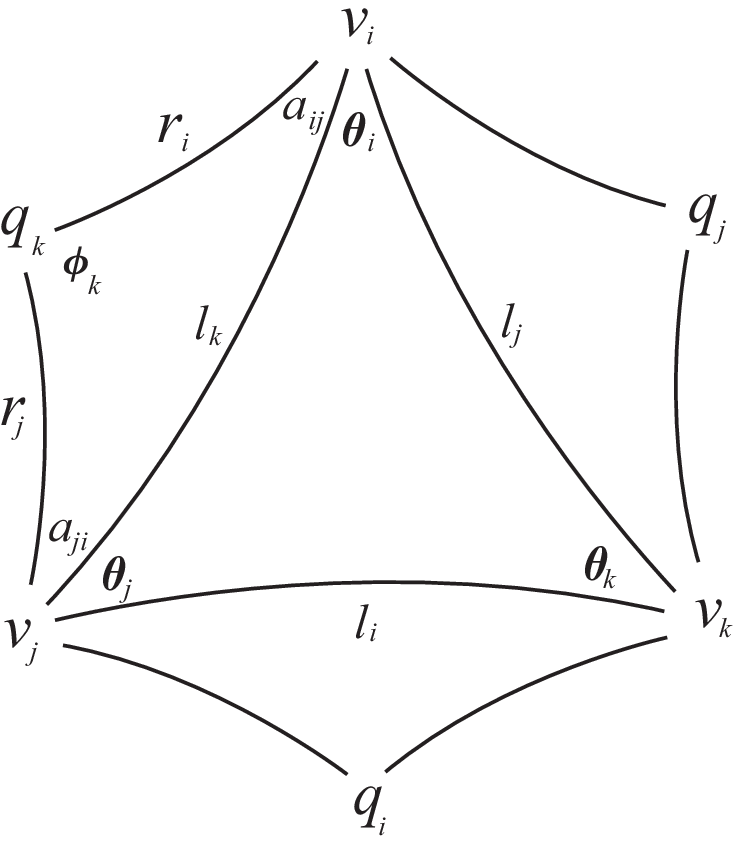}
\end{center}
\caption{\label{fig:packing-data}}

Let's recall the definition of
$\mathcal{M}_{\varepsilon,\delta}(\Phi(v_iv_j),\Phi(v_jv_k),\Phi(v_kv_i))$
in \S 1.4 and definition of
$\mathcal{D}_{\varepsilon,\delta}(\theta)$ in \S 3.

\begin{lemma}\label{thm:M} When $\varepsilon=0$ or $-1,$ we have
\begin{multline*}
\mathcal{M}_{\varepsilon,\delta}(\phi_1,\phi_2,\phi_3)
=\{(r_1,r_2,r_3)\in(J_{\varepsilon\delta})^3|
(r_i,r_j)\in\mathcal{D}_{\varepsilon,\delta}(\phi_k),
\{i,j,k\}=\{1,2,3\}\}.
\end{multline*}
\end{lemma}
\begin{proof} Given $\phi_k\in\mathring{I}_\delta,$ by Lemma \ref{thm:realize}, for $r_i,r_j\in
\mathcal{D}_{\varepsilon,\delta}(\phi_k),$ there is generalized
hyperbolic triangle $\triangle v_iv_jq_k$ of type
$(\varepsilon,\varepsilon,\delta)$ with edge lengths $r_i,r_j$ and
angle $\phi_k$ of type $\delta$ between the two edges. We obtain
the edge length $l_k=l(v_iv_j).$ There is a case which is not
contained in Lemma \ref{thm:realize}: $\delta=1, \phi_k=\pi.$ It
is easier since we have $l_k=r_i+r_j.$

For $\varepsilon=0,$ we get $l_i,l_j,l_k\in\mathbb{R}.$ Thus there
exists a decorated ideal triangle with three edges length
$l_i,l_j,l_k.$ For $\varepsilon=-1,$ the inequality defining
$\mathcal{D}_{-1,\delta}(\phi_k)$ guarantees that $l_i,l_j,l_k>0.$
Thus there exists a right-angled hexagon with three edges length
$l_i,l_j,l_k.$
\end{proof}

\subsection{A proof of Theorem \ref{thm:generalized-circle-packing} for $\varepsilon=0$}

We assume that indices $i,j,k$ are distinct in this section. First
by Lemma \ref{thm:realize} and Lemma \ref{thm:M}, we have
$\mathcal{M}_{0,\delta}(\phi_i,\phi_j,\phi_k)=\mathbb{R}^3.$
Therefore $\mathcal{N}_{0,\delta}(\Phi)=\mathbb{R}^V.$

The case $\varepsilon=0$ is very simple. Indeed, for a fixed
$\Phi:E\to I_\delta,$ there exists a map $C:V\to\mathbb{R}_{>0}$
so that for all $r\in\mathcal{N}_{0,\delta}(\Phi)=\mathbb{R}^V$,
\begin{align}\label{fml:d}
\widetilde{K}(r)(v)=C(v)e^{-r(v)},
\end{align}
i.e., the discrete curvature $\widetilde{K}(r)$ is uniformly
proportional to $e^{-r(v)}.$ Thus, one sees easily that the
generalized circle packing metric $r\in\mathbb{R}^V$ is determined
by its generalized discrete curvature $\widetilde{K}:V\to
\mathbb{R}$ and $\{\widetilde{K}(v)|v\in V\}=\mathbb{R}^V_{>0}$.

Indeed, in Figure \ref{fig:packing-data}, the generalized triangle
$\triangle v_iv_jq_k$ has type $(0,0,\delta)$, edge lengths
$r_i,r_j,l_k$ and inner angle $\phi_k$ opposite to $l_k$. By the
cosine law (Lemma \ref{thm:cosine-law} (\ref{fml:cosine2})), we
have
$$2\rho_\delta^2(\frac{\phi_k}2)=\frac{\tau_0'(l_k)}
{\tau_0(r_i)\tau_0(r_j)}=2e^{l_k-r_i-r_j}.$$ Thus
\begin{align}\label{fml:a}
 e^{l_k}=\rho_\delta^2(\frac{\phi_k}2)e^{r_i+r_j}.
\end{align}
By the cosine law (\ref{fml:cosine2}) for the decorated ideal
triangle $\triangle v_iv_jv_k$ (or type (0,0,0) generalized
triangle) we have
\begin{align}\label{fml:b}
\frac{\theta_k^2}4=e^{l_k-l_i-l_j}.
\end{align}
Combining (\ref{fml:a}) and (\ref{fml:b}) we obtain
$$\frac{\theta_k^2}4=e^{l_k-l_i-l_j}=\frac{\rho_\delta^2(\frac{\phi_k}2)e^{r_i+r_j}}
{\rho_\delta^2(\frac{\phi_i}2)e^{r_j+r_k}\rho_\delta^2(\frac{\phi_j}2)e^{r_k+r_i}}
=\frac{\rho_\delta^2(\frac{\phi_k}2)}{\rho_\delta^2(\frac{\phi_i}2)\rho_\delta^2(\frac{\phi_j}2)}e^{-2r_k}.$$
Or
\begin{align}\label{fml:c}
\theta_k=\frac{2\rho_\delta(\frac{\phi_k}2)}{\rho_\delta(\frac{\phi_i}2)\rho_\delta(\frac{\phi_j}2)}e^{-r_k}.
\end{align}
Summing up (\ref{fml:c}) for all triangle having $v_k$ as a vertex,
we obtain (\ref{fml:d}) where $C(v)$ depends explicitly on $\delta$.

\subsection{Rigidity of circle packing for $\varepsilon=-1$ or $(\varepsilon,\delta)=(1,1)$}

In this subsection we prove Theorem
\ref{thm:generalized-circle-packing} for the case $\varepsilon=-1$
and  give a new proof of Thurston's rigidity Theorem
\ref{thm:circle-packing}.

In the variational framework, the natural parameter is
$u=(u_1,u_2,u_3)$ where
\begin{align}\label{fml:e}
u_i=-\int_{r_i}^{\infty}\frac1{\tau_{\varepsilon\delta}(t)}dt
\end{align} for $\varepsilon=\pm1$.  Note that by the definition of
$\tau_s$, $\frac{\partial u_i}{\partial r_i} <0$.

\begin{lemma}\label{thm:moduli} Fix $\phi=(\phi_1,\phi_2,\phi_3)\in I_\delta^3,$ the set
$u(\mathcal{M}_{\varepsilon,\delta}(\phi_1,\phi_2,\phi_3))$ is an
open convex polyhedron in $\mathbb{R}^3$ in the following cases:
\begin{enumerate}
\item[(i)] $\varepsilon=-1$ or \\
\item[(ii)] $(\varepsilon,\delta)=(1,1)$ and
$\phi_i\in[\frac\pi2,\pi], i=1,2,3.$
\end{enumerate}
\end{lemma}

\begin{proof}
(i) When $\varepsilon=-1,$ we have figured out the set
$\mathcal{M}_{\varepsilon,\delta}(\phi_1,\phi_2,\phi_3)$ in Lemma
\ref{thm:M}.

If the type of $\phi_k$ is $\delta=1,$ by Lemma \ref{thm:realize},
we need $l_k >0$ which is the same as
\begin{align}\label{fml:moduli 1}
\cos \phi_k=\frac{-\cosh l_k +\sinh r_i\sinh r_j}{\cosh r_i\cosh
r_j}<\frac{-1 +\sinh r_i\sinh r_j}{\cosh r_i\cosh r_j}.
\end{align} In
this case $u_i=-\int_{r_i}^\infty\frac1{\cosh t}dt=2\arctan
e^{r_i}-\pi\in(-\frac\pi2,0).$ Then
\begin{align}\label{fml:moduli 2}
\sinh r_i=\frac1{\tan u_i},\  \cosh r_i=\frac1{\sin u_i}.
\end{align}
By substituting (\ref{fml:moduli 2}) into (\ref{fml:moduli 1}), we
obtain $\cos \phi_k<\cos(u_i+u_j).$ Since $\phi_k\in(0,\pi],$
$-u_i-u_j\in(0,\pi),$ we see that $l_k >0$ if and only if
$u_i+u_j>-\phi_k.$ Then
$u(\mathcal{M}_{\varepsilon,\delta}(\phi_1,\phi_2,\phi_3))
=\{(u_1,u_2,u_3)\in\mathbb{R}^3_{<0}|u_i+u_j>-\phi_k,\{i,j,k\}=\{1,2,3\}\}.$
The last description shows that
$u(\mathcal{M}_{\varepsilon,\delta}(\phi_1,\phi_2,\phi_3))$ is a
convex polyhedron.

If the type of $\phi_k$ is $\delta=0,$ we need $l_k >0$ which is
equivalent to
$$\frac{\phi_k^2}{2}=\frac{\cosh l_k +\cosh
(r_i-r_j)}{e^{r_i+r_j}}>\frac{1 +\cosh
(r_i-r_j)}{e^{r_i+r_j}}=\frac12(e^{-r_i}+e^{-r_j})^2.$$ Since in
this case $u_i=-\int_{r_i}^\infty\frac1{e^t}dt=-e^{r_i}<0,$, we
see that $l_k >0$ is the same as  have $u_i+u_j>-\phi_k.$ Thus
$u(\mathcal{M}_{\varepsilon,\delta}(\phi_1,\phi_2,\phi_3))=\{(u_1,u_2,u_3)\in\mathbb{R}^3_{<0}|
l_i >0$ for all $i$\} =$\{(u_1,u_2,u_3)\in\mathbb{R}^3_{<0} |
u_i+u_j>-\phi_k, \{i,j,k\}=\{1,2,3\}\}$. The last description
shows that
$u(\mathcal{M}_{\varepsilon,\delta}(\phi_1,\phi_2,\phi_3))$ is a
convex polyhedron.

If the type of $\phi_k$ is $\delta=-1,$ we need $l_k >0$ which is
the same as
\begin{align*}
\cosh \phi_k&=\frac{\cosh l_k +\cosh r_i\cosh r_j}{\sinh r_i\sinh
r_j}\\
&>\frac{1 +\cosh r_i\cosh r_j}{\sinh r_i\sinh r_j}\\
&=\frac12(\tanh\frac{r_i}2\tanh\frac{r_j}2+\frac1{\tanh\frac{r_i}2\tanh\frac{r_j}2}).
\end{align*}
Hence $l_k >0$ if and only if
$$e^{-\phi_k}<\tanh\frac{r_i}2\tanh\frac{r_j}2.$$
It follows that $l_k >0$ is the same as,
$$\ln\tanh\frac{r_i}2+\ln\tanh\frac{r_j}2>-\phi_k.$$
Since in this case $u_i=-\int_{r_i}^\infty\frac1{\sinh
t}dt=\ln\tanh\frac{r_i}2<0,$ we see that $l_k >0$ if and only if
$u_i+u_j>-\phi_k.$  It follows that
$u(\mathcal{M}_{\varepsilon,\delta}(\phi_1,\phi_2,\phi_3))=\{(u_1,u_2,u_3)\in\mathbb{R}^3_{<0}|u_i+u_j>-\phi_k,\{i,j,k\}=\{1,2,3\}\}.$
The last description shows that
$u(\mathcal{M}_{\varepsilon,\delta}(\phi_1,\phi_2,\phi_3))$ is a
convex polyhedron.

(ii) When $\varepsilon=1,\delta=1$ and $\phi_k\in[\frac\pi2,\pi]$,
Thurston observed that
$\mathcal{M}_{1,1}(\phi_1,\phi_2,\phi_3)=\mathbb{R}_{>0}^3.$
Indeed, due to $\phi_i\in[\frac\pi2,\pi],$ we see $l_i>r_j$.
Similarly, $l_j>r_i.$ Hence $l_i+l_j>r_j+r_i>l_k.$ This shows that
$l_1, l_2, l_3$ are the lengths of a hyperbolic triangle. In this
case $u_i=-\int_{r_i}^\infty\frac1{\sinh
t}dt=\ln\tanh\frac{r_i}2<0,$ Thus
$u(\mathcal{M}_{1,1}(\phi_1,\phi_2,\phi_3))=\mathbb{R}_{<0}^3.$
\end{proof}

\begin{lemma}\label{thm:symmetry} For $\varepsilon=\pm 1$ and fixed
$(\phi_1,\phi_2,\phi_3)\in I_\delta^3,$ the Jacobi matrix $A$ of
the function
$(\theta_1,\theta_2,\theta_3)=(\theta_1(u),\theta_2(u),\theta_3(u)):
u(\mathcal{M}_{\varepsilon,\delta}(\phi_1,\phi_2,\phi_3))\to\mathbb{R}^3$
is symmetric.
\end{lemma}

\begin{proof} Consider the $(\varepsilon,\varepsilon,\delta)-$triangle
$\triangle v_iv_jq_k$ in Figure \ref{fig:packing-data} with edge
lengths $r_i,r_j,l_k$ and opposite generalized angles
$a_{ji},a_{ij},\phi_k$. By applying the derivative cosine law
(\ref{fml:derivative2}) to $\triangle v_iv_jq_k$, the third row of
the matrix in (\ref{fml:derivative2}) is
$$d\phi_k=\frac{-\rho_\delta(\phi_k)}{\sqrt{-\det G_\theta}}
(\rho'_\varepsilon(a_{ji})dr_i+\rho'_\varepsilon(a_{ij})dr_j-dl_k).$$
Since $\phi_k$ is fixed, then $d\phi_k=0.$ We have
$dl_k=\rho'_\varepsilon(a_{ji})dr_i+\rho'_\varepsilon(a_{ij})dr_j.$
For $\varepsilon=\pm 1,$ we have
$u_i=-\int_{r_i}^\infty\frac1{\tau_{\varepsilon\delta}(t)}dt.$
Then $dr_i=\tau_{\varepsilon\delta}(r_i)du_i.$ Thus
\begin{align}
\left(
\begin{array}{ccc}
dl_1 \\
dl_2 \\
dl_3
\end{array}\right)&=\left(
\begin{array}{ccc}
0&\rho'_{\varepsilon}(a_{23})&\rho'_{\varepsilon}(a_{32}) \\
\rho'_{\varepsilon}(a_{13})&0&\rho'_{\varepsilon}(a_{31}) \\
\rho'_{\varepsilon}(a_{12})&\rho'_{\varepsilon}(a_{21})&0
\end{array}
\right)\left(
\begin{array}{ccc}
dr_1 \\
dr_2 \\
dr_3
\end{array}\right) \notag \\
&=\left(
\begin{array}{ccc}
0&\rho'_{\varepsilon}(a_{23})&\rho'_{\varepsilon}(a_{32}) \\
\rho'_{\varepsilon}(a_{13})&0&\rho'_{\varepsilon}(a_{31}) \\
\rho'_{\varepsilon}(a_{12})&\rho'_{\varepsilon}(a_{21})&0
\end{array}
\right) \left(
\begin{array}{ccc}
\tau_{\varepsilon\delta}(r_1)&0&0 \\
0&\tau_{\varepsilon\delta}(r_2)&0 \\
0&0&\tau_{\varepsilon\delta}(r_3)
\end{array}
\right)\left(
\begin{array}{ccc}
du_1 \\
du_2 \\
du_3
\end{array}\right). \label{fml:symmetry 1}
\end{align}

For the $(\varepsilon,\varepsilon,\varepsilon)-$triangle
$\triangle v_1v_2v_3$ with angles $\theta_1,\theta_2,\theta_3$ and
edge lengths $l_1,l_2,l_3$, by the derivative cosine law
(\ref{fml:derivative1}) and (\ref{fml:symmetry 1}), we have
\begin{align*}
\left(
\begin{array}{ccc}
d\theta_1 \\
d\theta_2 \\
d\theta_3
\end{array}\right)
&=\frac{-1}{\sqrt{-\det G_l}} \left(
\begin{array}{ccc}
\tau_1(l_1)&0&0 \\
0&\tau_1(l_2)&0 \\
0&0&\tau_1(l_3)
\end{array}
\right)
\left(
\begin{array}{ccc}
-1&\rho'_{\varepsilon}(\theta_3)&\rho'_{\varepsilon}(\theta_2) \\
\rho'_{\varepsilon}(\theta_3)&-1&\rho'_{\varepsilon}(\theta_1) \\
\rho'_{\varepsilon}(\theta_2)&\rho'_{\varepsilon}(\theta_1)&-1
\end{array}
\right)\\
&\left(
\begin{array}{ccc}
0&\rho'_{\varepsilon}(a_{23})&\rho'_{\varepsilon}(a_{32}) \\
\rho'_{\varepsilon}(a_{13})&0&\rho'_{\varepsilon}(a_{31}) \\
\rho'_{\varepsilon}(a_{12})&\rho'_{\varepsilon}(a_{21})&0
\end{array}
\right)\left(
\begin{array}{ccc}
\tau_{\varepsilon\delta}(r_1)&0&0 \\
0&\tau_{\varepsilon\delta}(r_2)&0 \\
0&0&\tau_{\varepsilon\delta}(r_3)
\end{array}
\right)\left(
\begin{array}{ccc}
du_1 \\
du_2 \\
du_3
\end{array}\right)\\
&=:\frac{-1}{\sqrt{-\det G_l}}N\left(
\begin{array}{ccc}
du_1 \\
du_2 \\
du_3
\end{array}\right).
\end{align*}

To show the Jacobi matrix $A$ of
$(\theta_1(u),\theta_2(u),\theta_3(u))$ is symmetric, we only need
to check that $N=(N_{ij})$ is symmetric. In fact
\begin{equation*}
N_{ij}=\tau_1(l_i)\tau_{\varepsilon\delta}(r_j)
(-\rho'_{\varepsilon}(a_{jk})+\rho'_{\varepsilon}(\theta_j)\rho'_{\varepsilon}(a_{ji})).
\end{equation*}
By cosine law (\ref{fml:cosine3}),
\begin{multline*}
N_{ij}=\tau_1(l_i)\tau_{\varepsilon\delta}(r_j)
(-\frac{\varepsilon\tau'_{\varepsilon\delta}(r_k)
+\tau'_{\varepsilon\delta}(r_j)\tau'_1(l_i)}
{\tau_{\varepsilon\delta}(r_j)\tau_1(l_i)}\\
+\frac{\varepsilon \tau'_1(l_j)+\tau'_1(l_i)\tau'_1(l_k)}
{\tau_1(l_i)\tau_1(l_k)} \frac{\varepsilon
\tau'_{\varepsilon\delta}(r_i)+\tau'_{\varepsilon\delta}(r_j)\tau'_1(l_k)}{\tau_{\varepsilon\delta}(r_j)\tau_1(l_k)})
\end{multline*}
\begin{multline*}
=\frac{\varepsilon \tau'_1(l_k)[-\tau_1^2(l_k)+
\tau'_{\varepsilon\delta}(r_j)\tau'_1(l_j)+\tau'_{\varepsilon\delta}(r_i)\tau'_1(l_i)]}
{\tau_1^2(l_k)}\\
+\frac{[\tau_1'^2(l_k)-\tau_1^2(l_k)]\tau'_{\varepsilon\delta}(r_j)\tau'_1(l_i)+\varepsilon^2
\tau'_{\varepsilon\delta}(r_i)\tau'_1(l_j)}{\tau_1^2(l_k)}.
\end{multline*}

Since $\varepsilon=\pm 1,
\tau_1(l_k)=\frac12e^{l_k}-\frac12e^{-l_k}=\sinh l_k,$ we have
$\tau_1'^2(l_k)-\tau_1^2(l_k)=1=\varepsilon^2.$ Therefore $N_{ij}$
is symmetric in the index $i,j.$
\end{proof}

\begin{lemma}\label{thm:definite} Given $(\phi_1,\phi_2,\phi_3)\in I_\delta^3,$ the
Jacobi matrix $A$ of
$(\theta_1,\theta_2,\theta_3)=(\theta_1(u),\theta_2(u),\theta_3(u)):
u(\mathcal{M}_{\varepsilon,\delta}(\phi_1,\phi_2,\phi_3))\to\mathbb{R}^3$
is negative definite in the following cases:
\begin{enumerate}
\item[(i)] $\varepsilon=-1$ or\\
\item[(ii)] \rm(Colin de Verdi\'ere \cite{cdv}, Chow-Luo
\cite{cl}) \it $(\varepsilon,\delta)=(1,1)$ and
$\phi_i\in[\frac\pi2,\pi], i=1,2,3$.
\end{enumerate}
\end{lemma}

We remark that the case (ii) for $\phi_i=\pi$ was first proved by
Colin de Verdi\'ere \cite{cdv}, and was proved for
$\phi_i\in[\frac\pi2,\pi]$ in \cite{cl} using the Maple program. Our
proof of (ii) is new.

\begin{proof} In the proof of Lemma \ref{thm:symmetry}, the Jacobi
matrix $A=\frac{-1}{\sqrt{-\det G_l}}N.$ Hence it is sufficient to
show $N$ is positive definite. First $\det N>0.$ Indeed $N$ is a
product of four matrixes. The first one and forth one are diagonal
matrixes with positive determinant. By Lemma
\ref{thm:determinant}, the determinant of the second matrix
(negative of an angle Gram matrix) is positive. The determinant of
the third matrix is
$\rho'_\varepsilon(a_{12})\rho'_\varepsilon(a_{23})\rho'_\varepsilon(a_{31})+
\rho'_\varepsilon(a_{32})\rho'_\varepsilon(a_{21})\rho'_\varepsilon(a_{13})>0.$

Since the set
$u(\mathcal{M}_{\varepsilon,\delta}(\phi_1,\phi_2,\phi_3))$ is
connected and the $\det N\neq 0,$ to show $N$ is positive
definite, we only need to check $N$ is positive definite at one
point in
$u(\mathcal{M}_{\varepsilon,\delta}(\phi_1,\phi_2,\phi_3))$. We
choose the point such that $r_1=r_2=r_3=r.$

\begin{center}
\includegraphics[scale=.7]{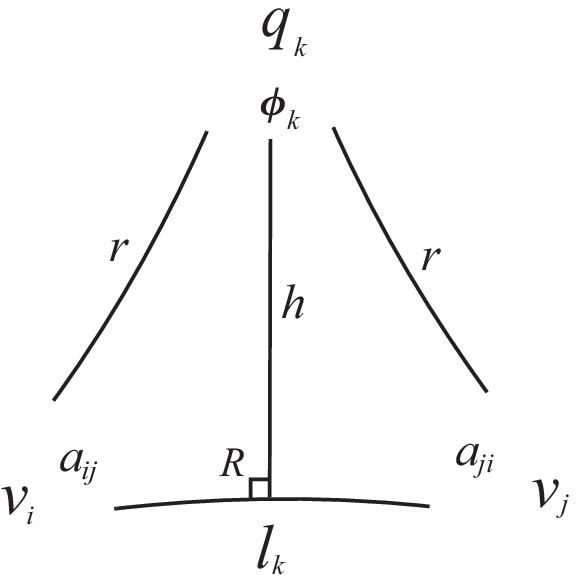}
\end{center}
\caption{\label{fig:special}}

Now in Figure \ref{fig:packing-data} and Figure \ref{fig:special},
in the $(\varepsilon,\varepsilon,\delta)-$triangle $\triangle
v_iv_jq_k$ with edge lengths $r,r,l_k$ and opposite generalized
angles $a_{ij}=a_{ji},\phi_k$, let $h$ be the geodesic length
realizing the distance between the edge $v_iv_j$ and the
generalized vertex $q_k$. In the $(\varepsilon,1,\delta)$-triangle
$\triangle v_iRq_k,$ where $Rq_k$ is perpendicular to $v_iv_j,$ by
the cosine law (\ref{fml:cosine3}), we have
$$\cos\frac\pi2=\frac{-\tau'_{\varepsilon\delta}(r)+\tau'_{\delta}(h)\tau'_\varepsilon(\frac{l_k}2)}
{\tau_{\delta}(h)\tau_\varepsilon(\frac{l_k}2)}.$$ Thus
$\tau'_{\delta}(h)=\tau'_{\varepsilon\delta}(r)/\tau'_\varepsilon(\frac{l_k}2)$.
Also in $\triangle v_iRq_k,$ by the cosine law
(\ref{fml:cosine3}), we have
$$
\rho'_{\varepsilon}(a_{ij})=\frac{-\varepsilon
\tau'_{\delta}(h)+\tau'_\varepsilon(\frac{l_k}2)\tau'_{\varepsilon\delta}(r)}
{\tau_\varepsilon(\frac{l_k}2)\tau_{\varepsilon\delta}(r)} =
\frac{\tau'_{\varepsilon\delta}(r)[-\varepsilon+\tau_\varepsilon'^2(\frac{l_k}2)]}{\tau'_\varepsilon(\frac{l_k}2)\tau(\frac{l_k}2)\tau_{\varepsilon\delta}(r)}.$$
Since
$$-\varepsilon+\tau_\varepsilon'^2(l)=-\varepsilon+(\frac12e^l+\frac12\varepsilon
e^{-l})^2=(\frac12e^l-\frac12\varepsilon
e^{-l})^2=\tau_\varepsilon^2(l),$$ then
\begin{align}\label{fml:definite 1}
\rho'_{\varepsilon}(a_{ij})=\frac{\tau'_{\varepsilon\delta}(r)\tau_\varepsilon(\frac{l_k}2)}
{\tau_{\varepsilon\delta}(r)\tau'_\varepsilon(\frac{l_k}2)}.
\end{align}

By substituting (\ref{fml:definite 1}) into $N$, let
$$N_1=\left(
\begin{array}{ccc}
-1&\rho'_{\varepsilon}(\theta_3)&\rho'_{\varepsilon}(\theta_2) \\
\rho'_{\varepsilon}(\theta_3)&-1&\rho'_{\varepsilon}(\theta_1) \\
\rho'_{\varepsilon}(\theta_2)&\rho'_{\varepsilon}(\theta_1)&-1
\end{array}
\right)\left(
\begin{array}{ccc}
0&\tau_\varepsilon(\frac{l_1}2)/\tau_\varepsilon'(\frac{l_1}2)
&\tau_\varepsilon(\frac{l_1}2)/\tau_\varepsilon'(\frac{l_1}2) \\
\tau_\varepsilon(\frac{l_2}2)/\tau_\varepsilon'(\frac{l_2}2)&0
&\tau_\varepsilon(\frac{l_2}2)/\tau_\varepsilon'(\frac{l_2}2) \\
\tau_\varepsilon(\frac{l_3}2)/\tau_\varepsilon'(\frac{l_3}2)
&\tau_\varepsilon(\frac{l_3}2)/\tau_\varepsilon'(\frac{l_3}2)&0
\end{array}
\right).$$ Then $N,N_1$ differ by left and right multiplication by
diagonal matrices of positive diagonal entries. We will show that
the determinants of the $1\times1, 2\times2$ principal submatrices
of $N_1$ are positive. This implies that the determinants of the
$1\times1, 2\times2$ principal submatrices of $N$ are positive.
Thus $N$ is positive definite.

Indeed, the determinant of the $1\times1$ principal submatrix of
$N_1$ is
\begin{align}\label{fml:definite 2}
\rho'_\varepsilon(\theta_3)\frac{\tau_\varepsilon(\frac{l_2}2)}{\tau_\varepsilon'(\frac{l_2}2)}
+\rho'_\varepsilon(\theta_2)\frac{\tau_\varepsilon(\frac{l_3}2)}{\tau_\varepsilon'(\frac{l_3}2)}.
\end{align} And the determinant of the $2\times2$ principal matrix of $N_1$ is
\begin{multline}\label{fml:definite 3}
[\rho_\varepsilon'^2(\theta_3)-1]
\frac{\tau_\varepsilon(\frac{l_1}2)}{\tau_\varepsilon'(\frac{l_1}2)}
\frac{\tau_\varepsilon(\frac{l_2}2)}{\tau_\varepsilon'(\frac{l_2}2)}
+[\rho'_\varepsilon(\theta_2)+\rho'_\varepsilon(\theta_1)\rho'_\varepsilon(\theta_3)]
\frac{\tau_\varepsilon(\frac{l_2}2)}{\tau_\varepsilon'(\frac{l_2}2)}
\frac{\tau_\varepsilon(\frac{l_3}2)}{\tau_\varepsilon'(\frac{l_3}2)}\\
+[\rho'_\varepsilon(\theta_1)+\rho'_\varepsilon(\theta_2)\rho'_\varepsilon(\theta_2)]
\frac{\tau_\varepsilon(\frac{l_1}2)}{\tau_\varepsilon'(\frac{l_1}2)}
\frac{\tau_\varepsilon(\frac{l_3}2)}{\tau_\varepsilon'(\frac{l_3}2)}.
\end{multline}

(i) If $\varepsilon=-1$, then
$\rho_\varepsilon'(\theta_i)=\cosh\theta_i.$ Each term in
(\ref{fml:definite 2}) and (\ref{fml:definite 3}) is positive.
Hence (\ref{fml:definite 2}) and (\ref{fml:definite 3}) are
positive.

(ii) If $\varepsilon=1$, then
$\rho_\varepsilon'(\theta_i)=\cos\theta_i.$ We see that the
expression in (\ref{fml:definite 2}) is positive is the same as
\begin{equation}\label{fml:1}
\cos\theta_3\tanh\frac{l_2}2+\cos\theta_2\tanh\frac{l_3}2>0.
\end{equation}
By cosine law, (\ref{fml:1}) is equivalent to
\begin{equation}\label{fml:2}
\frac{-\cosh l_3 +\cosh l_1\cosh l_2}{\sinh l_1\sinh
l_2}\frac{\sinh l_2}{1+\cosh l_2} +\frac{-\cosh l_2 +\cosh
l_1\cosh l_3}{\sinh l_1\sinh l_3}\frac{\sinh l_3}{1+\cosh l_3}>0.
\end{equation}
By calculation, (\ref{fml:2}) is equivalent to
\begin{equation}\label{fml:3}
\cosh l_1>\frac{\cosh^2 l_2+\cosh^2 l_3+\cosh l_2+\cosh
l_3}{2\cosh l_2\cosh l_3+\cosh l_2+\cosh l_3}.
\end{equation}
Since $l_1>|l_2-l_3|,$ therefore $\cosh l_1>\cosh (l_2-l_3).$ To
show (\ref{fml:3}) holds, it is enough to check
\begin{equation}\label{fml:4}
\cosh l_2\cosh l_3-\sinh l_2\sinh l_3\geq \frac{\cosh^2
l_2+\cosh^2 l_3+\cosh l_2+\cosh l_3}{2\cosh l_2\cosh l_3+\cosh
l_2+\cosh l_3}.
\end{equation}

We simplify the notation by introducing $a:=\cosh l_2>1, b:=\cosh
l_3>1.$ Then (\ref{fml:4}) is rewritten as
\begin{equation}\label{fml:5}
ab-\sqrt{(a^2-1)(b^2-1)}\geq\frac{a^2+b^2+a+b}{2ab+a+b}.
\end{equation}
(\ref{fml:5}) is equivalent to
\begin{equation}\label{fml:6}
ab-\frac{a^2+b^2+a+b}{2ab+a+b}\geq\sqrt{(a^2-1)(b^2-1)}.
\end{equation}

Since $a>1,b>1,$ the left hand side of (\ref{fml:6}) is positive.
To show (\ref{fml:6}) holds, we square the two sides and simplify.
We have
\begin{align*}
&(ab^4+a^4b-a^3b^2-a^2b^3)+(a^4+b^4-2a^2b^2)+(a^3+b^3-ab^2-a^2b)\geq0\\
\Leftrightarrow& (ab+1)(a^3+b^3-a^2b-ab^2)+(a^2-b^2)^2\geq0\\
\Leftrightarrow& (ab+1)(a+b)(a-b)^2+(a^2-b^2)^2\geq0.
\end{align*}
This shows that the expression in (\ref{fml:definite 2}) is
positive.

We see that the expression in (\ref{fml:definite 3}) is positive
is the same as
\begin{multline}\label{fml:7}
-\sin^2 \theta_3\tanh\frac{l_1}2\tanh\frac{l_2}2+
[\cos\theta_2+\cos\theta_1\cos\theta_3]\tanh\frac{l_2}2\tanh\frac{l_3}2\\
+[\cos\theta_1+\cos\theta_2\cos\theta_3]\tanh\frac{l_1}2\tanh\frac{l_3}2>0.
\end{multline}
By the cosine law, (\ref{fml:7}) is equivalent to
\begin{multline}\label{fml:8}
-\sin^2 \theta_3\tanh\frac{l_1}2\tanh\frac{l_2}2+ \cosh
l_2\sin\theta_1\sin\theta_3\tanh\frac{l_2}2\tanh\frac{l_3}2\\
+\cosh
l_1\sin\theta_2\sin\theta_3\tanh\frac{l_1}2\tanh\frac{l_3}2>0.
\end{multline}
By the sine law, (\ref{fml:8}) is equivalent to
\begin{align*}
&(\cosh l_1\sinh l_2\tanh\frac{l_1}2+\cosh l_2\sinh
l_1\tanh\frac{l_2}2)\tanh\frac{l_3}2> \sinh
l_3\tanh\frac{l_1}2\tanh\frac{l_2}2\\
\Leftrightarrow& \frac{\cosh l_1\sinh
l_2}{\tanh\frac{l_2}2}+\frac{\cosh l_2\sinh
l_1}{\tanh\frac{l_1}2}> \frac{\sinh l_3}{\tanh\frac{l_3}2}\\
\Leftrightarrow& \cosh l_1(1+\cosh l_2)+\cosh l_2(1+\cosh
l_1)>1+\cosh l_3 \\
\Leftrightarrow& (\cosh l_1+\cosh l_2-1)+(2\cosh l_1\cosh
l_2-\cosh l_3)>0.
\end{align*} This is true since $\cosh l_1+\cosh l_2>1$ and $2\cosh
l_1\cosh l_2>\cosh l_1\cosh l_2+\sinh l_1\sinh
l_2=\cosh(l_1+l_2)>\cosh l_3.$ Thus shows that the expression in
(\ref{fml:definite 3}) is positive.
\end{proof}

\begin{corollary}\label{thm:closed-form-packing} Given $(\phi_1,\phi_2,\phi_3)\in I_\delta^3,$
the differential 1-form $\sum_{i=1}^3\theta_idu_i$ is closed in
$u(\mathcal{M}_{\varepsilon,\delta}(\phi_1,\phi_2,\phi_3)).$
Furthermore, its integration $w(u)=\int^u\sum_{i=1}^3\theta_idu_i$
is a strictly concave function defined in
$u(\mathcal{M}_{\varepsilon,\delta}(\phi_1,\phi_2,\phi_3))$ if
either
\begin{enumerate}
\item[(i)] $\varepsilon=-1$ or \\
\item[(ii)] $(\varepsilon,\delta)=(1,1)$ and
$\phi_i\in[\frac\pi2,\pi], i=1,2,3.$
\end{enumerate}
Furthermore
\begin{align}\label{fml:gradient}
\frac{\partial w}{\partial u_i} =\theta_i.
\end{align}
\end{corollary}

Let's prove the rigidity of circle packing for $\varepsilon=-1$,
or $(\varepsilon,\delta)=(1,1)$ and $\Phi:E\to[\frac\pi2,\pi].$
Fix $\Phi:E\to I_\delta.$ For each $r\in
\mathcal{N}_{\varepsilon,\delta}(\Phi),$ define the function $u$
in term of $r$ as in (\ref{fml:e}) by
$$u(v)=-\int_{r(v)}^\infty\frac1{\tau_{\varepsilon\delta}(t)}dt.$$

The set of all values of $u$ is
$u(\mathcal{N}_{\varepsilon,\delta}(\Phi))$ which is an open
convex set in $\mathbb{R}^V$ due to Lemma \ref{thm:moduli}. To be
more precisely, let us label vertices $V=\{v_1,...,v_n\}$ and use
$\{i,j,k\}\in F$ to denote a generalized triangle in $F$ with
vertices $v_i,v_j,v_k.$ Let $u_i=u(v_i).$ Then
$u(\mathcal{N}_{\varepsilon,\delta}(\Phi))=\{u\in\mathbb{R}^V|(u_i,u_j,u_k)
\in u(\mathcal{M}_{\varepsilon,\delta}(\phi_i,\phi_j,\phi_k))$ if
$\{i,j,k\}\in F$\} is convex since it is the intersection of the
convex set $\prod_{\{i,j,k\}\in
F}u(\mathcal{M}_{\varepsilon,\delta}(\phi_i,\phi_j,\phi_k))$ with
affine spaces.

We now use the function in Corollary \ref{thm:closed-form-packing}
to introduce a function $W:
u(\mathcal{N}_{\varepsilon,\delta}(\Phi))\to\mathbb{R}$ by
defining
$$W(u)=\sum_{\{i,j,k\}\in F} w(u_i, u_j, u_k)$$ where the sum is
over all triangles in $F$ with vertices $\{v_i,v_j,v_k\}.$ By
Corollary \ref{thm:closed-form-packing}, $W$ is smooth and
strictly concave down in
$u(\mathcal{N}_{\varepsilon,\delta}(\Phi))$ so that
$$\frac{\partial W}{\partial u_i} =\widetilde{K}(v_i)$$ by (\ref{fml:gradient})
and the definition of $\widetilde{K}.$

By Lemma \ref{thm:convex}, the map $\nabla
W:u(\mathcal{N}_{\varepsilon,\delta}(\Phi))\to\mathbb{R}^V$ is a
smooth embedding. Thus the map from
$\{r\in\mathcal{N}_{\varepsilon,\delta}(\Phi)\}$ to
$\{\widetilde{K}\in\mathbb{R}^V_{>0}\}$ is a smooth injective map.

\subsection{The image of $\widetilde{K}$ for $\varepsilon=-1$}

To prove that the image $X=\{\widetilde{K}(v)|v\in
V\}=\mathbb{R}^V_{>0}$ for $\varepsilon=-1$, we will show that $X$
is both open and closed in $\mathbb{R}^V_{>0}.$ By definition,
$X\subset \mathbb{R}^V_{>0}$ and $X$ is open due to the
injectivity of the map from
$\{r\in\mathcal{N}_{-1,\delta}(\Phi)\}$ to $\mathbb{R}^V_{>0}.$ It
remains to prove that $X$ is closed in $\mathbb{R}^V_{>0}.$ To
this end, let us first establish the following lemma.

\begin{lemma}\label{thm:infinity-packing} Let $\varepsilon=-1$ and $(\phi_i,\phi_j,\phi_k)$ be
given so that
$(\theta_i,\theta_j,\theta_k)=(\theta_i(r_i,r_j,r_k),$
$\theta_j(r_i,r_j,r_k),\theta_k(r_i,r_j,r_k))$ are considered as
functions of $r_i,r_j,r_k\in J_{-\delta}.$ Then
$$\lim_{r_k\to\infty}\theta_k(r_i,r_j,r_k)=0$$ and the convergence is
uniform.
\end{lemma}

\begin{proof} We only need to check that for
constants $a,b,c\in J_{-\delta}$ the following holds.

(1) If $\lim r_i=a,\lim r_j=b,\lim r_k=\infty,$ then $\lim
\theta_k=0.$

(2) If $\lim r_i=c,\lim r_j=\infty,\lim r_k=\infty,$ then $\lim
\theta_k=0.$

(3) If $\lim r_i=\infty,\lim r_j=\infty,\lim r_k=\infty,$ then
$\lim \theta_k=0.$

The strategy of the proof is the same for all three cases. First,
in $\triangle P_jP_kQ_i,$ $\triangle P_kP_iQ_j,$ $\triangle
P_iP_jQ_k$ of type $(-1,-1,\delta)$, by the cosine law
(\ref{fml:cosine2}), the length $l_i,l_j,l_k$ can be write as
functions of $r_i,r_j,r_k:$
\begin{align*}
\cosh
l_i&=2\rho_\delta^2(\frac{\phi_i}2)\tau_{-\delta}(r_j)\tau_{-\delta}(r_k)-\frac12(e^{r_j-r_k}+e^{r_k-r_j}),\\
\cosh
l_j&=2\rho_\delta^2(\frac{\phi_j}2)\tau_{-\delta}(r_k)\tau_{-\delta}(r_i)-\frac12(e^{r_k-r_i}+e^{r_i-r_k}),\\
\cosh
l_k&=2\rho_\delta^2(\frac{\phi_k}2)\tau_{-\delta}(r_i)\tau_{-\delta}(r_j)-\frac12(e^{r_i-r_j}+e^{r_j-r_i}).
\end{align*}
When the limits of $r_i,r_j,r_k$ are given, we can find the limits
of $l_i,l_j,l_k.$ Then in $\triangle P_iP_jP_k$ (of type
$(-1,-1,-1)$), by the cosine law (\ref{fml:cosine3}), write
$\theta_k$ as a function of $l_i,l_j,l_k:$
$$\cosh \theta_k=\frac{\cosh l_k +\cosh
l_i\cosh l_j}{\sinh l_i\sinh l_j}.
$$
Then we find the limit of $\theta_k.$

(1) If $\lim r_i=a,\lim r_j=b,\lim r_k=\infty,$ then
\begin{align*}
\lim \cosh l_i&=\lim
[2\rho_\delta^2(\frac{\phi_i}2)\tau_{-\delta}(b)\frac12e^{r_k}-\frac12e^{r_k-b}]\\
&=\lim
\frac12e^{r_k}[2\rho_\delta^2(\frac{\phi_i}2)\tau_{-\delta}(b)-1]
\end{align*}
Since $\cosh l_i>0,$ then
$2\rho_\delta^2(\frac{\phi_i}2)\tau_{-\delta}(b)-1>0.$ When $\lim
r_k=\infty$, we have $\lim \cosh l_i=\infty.$

By symmetry $\lim \cosh l_j=\infty.$ Furthermore $\lim \cosh l_k$ is
finite. Therefore $\lim l_i=\lim l_j=\infty,$ $\lim l_k$ is finite.
Hence
$$\lim\cosh \theta_k=\lim\frac{\cosh l_k}{\sinh
l_i\sinh l_j}+1=1.$$ Therefore $\lim \theta_k=0.$

(2) If $\lim r_i=c,\lim r_j=\infty,\lim r_k=\infty,$ then
\begin{align*} \lim\cosh
l_i&=\lim [2\rho_\delta^2(\frac{\phi_i}2)\frac14e^{r_j+r_k}-\frac12(e^{r_j-r_k}+e^{r_k-r_j})]\\
&=\lim \frac12e^{r_j+r_k}(\rho_\delta^2(\frac{\phi_i}2)-e^{-2r_k}-e^{-2r_j})\\
&=\lim \frac12e^{r_j+r_k}\rho_\delta^2(\frac{\phi_i}2)\\
&=\infty,\\
\lim\cosh l_j&=\lim[2\rho_\delta^2(\frac{\phi_j}2)\tau_{-\delta}(c)\frac12e^{r_k}-\frac12e^{r_k-c}]\\
&=\lim
e^{r_k}(\rho_\delta^2(\frac{\phi_j}2)\tau_{-\delta}(c)-\frac12e^{-c})\\
&=\infty.
\end{align*}
Here we use the same argument in (1).

By the same calculation of $\lim\cosh l_j$, we see that $\lim\cosh
l_k=\lim e^{r_j}c_k=\infty$ for some constant $c_k.$ Hence
\begin{align*}
\lim\cosh \theta_k
&=\lim\frac{\cosh l_k}{\sinh l_i\sinh l_j}+1\\
&=\lim\frac{\cosh l_k}{\cosh l_i\cosh l_j}+1\\
&=\lim \frac{e^{r_j}c_k}{e^{r_k}c_je^{r_j+r_k}c_i}+1\\
&=1.
\end{align*} Therefore $\lim
\theta_k=0.$

(3) If $\lim r_i=\infty,\lim r_j=\infty,\lim r_k=\infty,$ by the
same calculation of $\lim\cosh l_i$ in (2), we see that $$\lim
\cosh l_k=\lim e^{r_i+r_j}a_k, \lim \cosh l_j=\lim e^{r_k+r_i}a_j,
\lim \cosh l_i=\lim e^{r_j+r_k}a_i$$ for some constants
$a_i,a_j,a_k.$ Hence
\begin{align*}
\lim\cosh \theta_k
&=\lim\frac{\cosh l_k}{\sinh l_i\sinh l_j}+1\\
&=\lim\frac{\cosh l_k}{\cosh l_i\cosh l_j}+1\\
&=\lim \frac{e^{r_i+r_j}a_k}{e^{r_k+r_i}a_je^{r_j+r_k}a_i}+1\\
&=1.
\end{align*}
 Therefore $\lim \theta_k=0.$
\end{proof}

To show that $X$ is closed in $\mathbb{R}^V_{>0},$ take a sequence
of radius $r^{(m)}$ in $\mathcal{N}_{\Phi,-1}$ such that
$\lim_{m\to\infty} \widetilde{K}^{(m)}\in\mathbb{R}^V_{>0}.$ To
prove the closeness, it is sufficient to show that there is a
subsequence, say $r^{(m)}$, so that $\lim_{m\to\infty} r^{(m)}$ is
in $\mathcal{N}_{-1,\delta}(\Phi)$.

Suppose otherwise, there is a subsequence, say $r^{(m)}$ so that
$\lim_{m\to\infty} r^{(m)}$ is in the boundary of
$\mathcal{N}_{-1,\delta}(\Phi)$. For $\delta=\pm 1,$ there are two
possibilities that either for some $v \in V$, $\lim_{m\to\infty}
r^{(m)}(v) =\infty$ or there is an edge $e$ such that
$l^{(m)}_e=0.$ For $\delta=0,$ although $r^{(m)}$'s are allowed to
be negative, the limit of $r^{(m)}$ can not be $-\infty$ since
when $\phi_k$ is given, the condition in Lemma \ref{thm:realize}
(3)
$$\phi_k>\exp(r_i^{(m)})+\exp(r_j^{(m)})$$
implies that $r_i^{(m)}$ is bounded away from $-\infty.$ Therefore
there are only those two possibilities as in case $\delta=\pm 1.$

In the first possibility, by Lemma \ref{thm:infinity-packing} we
see each generalized angle incident to the vertex $v$ converges to
$0$. Hence $\lim_{m\to\infty} \widetilde{K}^{(m)}(v)=0.$ This
contradicts the assumption that $\lim_{m\to\infty}
\widetilde{K}^{(m)}\in\mathbb{R}^V_{>0}.$

In the second possibility, in a hyperbolic right-angled hexagon
with lengths $l_i, l_j, l_k$ and opposite generalized angle
$\theta_i, \theta_j, \theta_k$, by cosine law we see that
\begin{align*}
\cosh\theta_j &= \frac{ \cosh l_j +
\cosh l_i\cosh l_k}{\sinh l_i \sinh l_k}\\
& >  \frac{ \cosh l_i
\cosh l_k}{\sinh l_i \sinh l_k }\\
& \geq  \frac{ \cosh l_i }{\sinh l_i }.
\end{align*}
Hence we have $\lim_{l_i \to 0} \theta_j =\infty.$ Hence the
generalized discrete curvature containing $\theta_j$ converges to
$\infty.$ This contradicts the assumption that $\lim_{m\to\infty}
\widetilde{K}^{(m)}\in\mathbb{R}^V_{>0}.$

\section{A proof of Theorem \ref{thm:generalized-circle-pattern}}

The proof is based on constructing an strictly concave energy
function on the space of all generated hyperbolic triangles of
type $(\varepsilon,\varepsilon,\delta)$ so that its gradient is
the generalized angles. Then using Lemma \ref{thm:convex} on
injectivity of gradient, we establish Theorem
\ref{thm:generalized-circle-pattern}.

\subsection{An energy functional on the space of triangles}

Fix a type $(\varepsilon, \varepsilon, \delta)\in\{-1,0,1\}^3.$
Consider $(\varepsilon, \varepsilon, \delta)$ type generalized
hyperbolic triangles whose edge length are $l_1,l_2,l_3$ and
opposite angles $\theta_1,\theta_2,\theta_3$ so that the angle
$\theta_i$ faces $l_i$ and the type of $\theta_3$ angle is
$\delta.$ For a fixed angle $\theta_3$, all values of the two edge
lengthes $(l_1,l_2)$ form the set
$\mathcal{D}_{\varepsilon,\delta}(\theta_3)$. For the definition
of $\mathcal{D}_{\varepsilon,\delta}(\theta_3)$, see \S 3.

For $h\in\mathbb{R},$ make a change of variables $(l_1,l_2)$ to
$(w_1,w_2)$ and $(\theta_1,\theta_2)$ to $(a_1,a_2)$ as follows.
Let $i=1,2,$
\begin{align}\label{fml:pattern u}
a_i=\int_1^{\theta_i}\rho_\varepsilon^h(t)dt
\end{align}
where $\rho_\varepsilon(t)=\int_0^t\cos(\sqrt{\varepsilon} x)dx$
and
\begin{align}\label{fml:pattern w}
w_i=\int_1^{l_i}\tau_{\varepsilon\delta}^{h-1}(t)dt
\end{align}
where
$\tau_{\varepsilon\delta}(t)=\frac12e^t-\frac12\varepsilon\delta
e^{-t}$ as introduced in \S 3. By the construction, the maps
$(l_1,l_2)$ to $w=(w_1,w_2)$ and $(\theta_1,\theta_2)$ to
$a=(a_1,a_2)$ are diffeomorphisms. Thus the cosine law relating
$l$ to $\theta$ can be considered, with $\theta_3$ fixed, as a
smooth map $a=a(w)$ defined on
$w(\mathcal{D}_{\varepsilon,\delta}(\theta_3))$

\begin{lemma}\label{thm:closed-form-pattern} Under the above
assumption, for a fixed angle $\theta_3,$ and any
$h\in\mathbb{R},$ the differential 1-form $a_1dw_2+a_2dw_1$ is
closed in $w(\mathcal{D}_{\varepsilon,\delta}(\theta_3)).$
Furthermore, the integration
$$F_{\theta_3,h}(w_1,w_2)=\int_{(1,1)}^{(w_1,w_2)}(a_1dw_2+a_2dw_1)$$ is strictly
concave down in $w(\mathcal{D}_{\varepsilon,\delta}(\theta_3))$.
In particular
\begin{align}\label{fml:pattern derivative}
\frac{\partial F_{\theta_3,h}}{\partial
w_i}=a_j=\int_1^{\theta_j}\rho_\varepsilon^h(t)dt
\end{align} for $\{i,j\}=\{1,2\}.$
\end{lemma}

\begin{proof} If $\theta_3$ is fixed, then $d\theta_3=0.$ By the derivative cosine law (\ref{fml:derivative1}) we have
\begin{align*}
\left(
\begin{array}{ccc}
dl_1 \\
dl_2
\end{array}\right)
&=\frac{-1}{\sqrt{-\det G_l}} \left(
\begin{array}{ccc}
\tau_{\varepsilon\delta}(l_1)&0 \\
0&\tau_{\varepsilon\delta}(l_2)
\end{array}
\right) \left(
\begin{array}{ccc}
\varepsilon&\tau_{\varepsilon\varepsilon}'(l_3) \\
\tau_{\varepsilon\varepsilon}'(l_3)&\varepsilon
\end{array}
\right) \left(
\begin{array}{ccc}
d\theta_1 \\
d\theta_2
\end{array}\right)\\
&=\frac{-1}{\sqrt{-\det G_l}} \left(
\begin{array}{ccc}
\tau_{\varepsilon\delta}(l_1)&0 \\
0&\tau_{\varepsilon\delta}(l_2)
\end{array}
\right) \left(
\begin{array}{ccc}
\tau_{\varepsilon\varepsilon}'(l_3)&\varepsilon \\
\varepsilon&\tau_{\varepsilon\varepsilon}'(l_3)
\end{array}
\right) \left(
\begin{array}{ccc}
d\theta_2 \\
d\theta_1
\end{array}\right).
\end{align*}
Since $dw_i=\tau_{\varepsilon\delta}^{h-1}(l_i)dl_i,
da_i=\rho_\varepsilon^h(\theta_i)d\theta_i,$ for $i=1,2,$ then
\begin{multline*}
\left(
\begin{array}{ccc}
dw_1 \\
dw_2
\end{array}\right)=\frac{-1}{\sqrt{-\det G_l}}
\left(
\begin{array}{ccc}
\tau_{\varepsilon\delta}^h(l_1)&0 \\
0&\tau_{\varepsilon\delta}^h(l_2)
\end{array}
\right) \left(
\begin{array}{ccc}
\tau_{\varepsilon\varepsilon}'(l_3)&\varepsilon \\
\varepsilon&\tau_{\varepsilon\varepsilon}'(l_3)
\end{array}
\right)\\ \left(
\begin{array}{ccc}
\rho_\varepsilon^{-h}(\theta_2)&0 \\
0&\rho_\varepsilon^{-h}(\theta_1)
\end{array}
\right)\left(
\begin{array}{ccc}
da_2 \\
da_1
\end{array}\right)
\end{multline*}
$$=:\frac{-1}{\sqrt{-\det G_l}}A\left(
\begin{array}{ccc}
da_2 \\
da_1
\end{array}\right).$$

Since
$A_{12}=\tau_{\varepsilon\delta}^h(l_1)\varepsilon\rho_\varepsilon^{-h}(\theta_1)
=\tau_{\varepsilon\delta}^h(l_2)\varepsilon\rho_\varepsilon^{-h}(\theta_2)=A_{21}$
by the sine law (\ref{fml:sine}), the matrix $A$ is symmetric.
Thus the differential 1-form $a_1dw_2+a_2dw_1$ is closed.
Therefore the function $F_{\theta_3,h}(w_1,w_2)$ is well defined.

The above calculation shows that the Hessian  of
$F_{\theta_3,h}(w_1,w_2)$ is the matrix $-\sqrt{-\det G_l}A^{-1}.$
To show the function $F_{\theta_3,h}(w_1,w_2)$ is strictly concave
down, we need to check that $A^{-1}$ is positive definite. It is
equivalent to show that $A$ is positive definite. By forgetting
the two diagonal matrices, it is enough to show
$$B=\left(
\begin{array}{ccc}
\tau_{\varepsilon\varepsilon}'(l_3)&\varepsilon \\
\varepsilon&\tau_{\varepsilon\varepsilon}'(l_3)
\end{array}
\right)$$ is positive definite. Since
$\tau_{\varepsilon\varepsilon}'(l_3) =1/2(e^{l_3} + \epsilon^2
e^{-l_3})>0$ and
$$\det
B=\tau_{\varepsilon\varepsilon}'^2(l_3)-\varepsilon^2=(\frac12e^{l_3}+\frac12\varepsilon\varepsilon
e^{-l_3})^2-\varepsilon^2=(\frac12e^{l_3}-\frac12\varepsilon\varepsilon
e^{-l_3})^2>0,$$ the matrix $B$ is positive definite.
\end{proof}

\subsection{A proof of Theorem \ref{thm:generalized-circle-pattern}}

Now the proof of Theorem \ref{thm:generalized-circle-pattern}
follows from the routine variational framework. Let us recall the
set up in $\S$1.5. Suppose $(\Sigma,G)$ is a cell decomposed
surface so that the sets of all vertices, edges and 2-cells are
$V,E,F$ respectively. The dual decomposition is $G^*$ with
vertices $V^*(\cong F)$. Elements in $V^*$ are denoted by $f^*$
where $f\in F.$ A quadrilateral $(v,v',f^*,f'^*)\in V\times
V\times V^* \times V^*$ in $\Sigma$ satisfies $vv'\in E, f>vv'$
and $f'>vv'.$

Now fix a type $(\varepsilon, \varepsilon, \delta)$ and a function
$\theta: E\to \mathring{I}_\delta$. The set of all circle pattern
metrics is
$\mathcal{E}_{\varepsilon,\delta}(\theta)=\{r\in(J_{\varepsilon\delta})^{V^*}|(r_i,r_j)\in\mathcal{D}_{\varepsilon,\delta}(\theta(v_iv_j))$
whenever $f_i,f_j$ share an edge\}. For any circle pattern metric
$r\in\mathcal{E}_{\varepsilon,\delta}(\theta),$ and a
quadrilateral $(v,v',f^*,f'^*)\in V\times V\times V^* \times V^*$
where $vv'\in E, f>vv'$ and $f'>vv',$ construct a type
$(\varepsilon, \varepsilon, \delta)$ generalized hyperbolic
triangle $\triangle f^*f'^*v$ so that the length of $f^*v$ and
$f'^*v$ are $r(f^*),r(f'^*),$ and the generalized angle at $v$ is
$\theta(vv')$ of type $\delta.$ Realize the quadrilateral
$(v,v',f^*,f'^*)$ geometrically as the metric double of $\triangle
f^*vf'^*$ along the edge $f^*f'^*.$ Now isometrically glue all
these geometric quadrilateral along edges. The result is a
polyhedral surface. Recall that for $h\in\mathbb{R},$ the
$K_h-$curvature of $r$
$$K_h:V^*\to\mathbb{R}$$ is defined by (\ref{fml:K_h}).

For $h\in\mathbb{R},$ make a change of parameter from $r\in
\mathcal{E}_{\varepsilon,\delta}(\theta)$ to $w=w(r)\in
w(\mathcal{E}_{\varepsilon,\delta}(\theta))$ so that
$w(r)(x)=\int_1^{r(x)}\tau_{\varepsilon\delta}^{h-1}(s)ds$ is
given by (\ref{fml:pattern w}). We now use Lemma
\ref{thm:closed-form-pattern} to construct a smooth strictly
concave function
$W:w(\mathcal{E}_{\varepsilon,\delta}(\theta))\to\mathbb{R}$ so
that $\nabla W|_w$ is the generalized curvature $K_h$ of $r$ where
$w=w(r).$ Then using Lemma \ref{thm:convex}, we see that Theorem
\ref{thm:generalized-circle-pattern} follows.

Here is the construction. For $w=w(r)\in
w(\mathcal{E}_{\varepsilon,\delta}(\theta))$ and each
quadrilateral $(v,v',f^*,f'^*)$ in $\Sigma$, we define the
F-energy of it in $r$ metric to be
$$F_{\theta(vv'),h}(w(r(f^*)),w(r(f'^*)))$$ where $F_{\theta,h}$ is
given by Lemma \ref{thm:closed-form-pattern}. The function
$W:w=w(r)\in w(\mathcal{E}_{\varepsilon,\delta}(\theta))\to
\mathbb{R}$ is the sum of F-energies of all quadrilateral
$(v,v',f^*,f'^*)$ in $r$ metric. By the construction, $W$ is
smooth and strictly concave. By (\ref{fml:pattern derivative}) we
have $$\nabla W|_w=K_h|_r$$ where $w=w(r).$ This ends the proof.

\newpage

\section*{Appendix A. Formulas of cosine and sine laws}

\parpic(0cm,0cm)(0cm,4cm)[r]{\includegraphics[scale=.53]{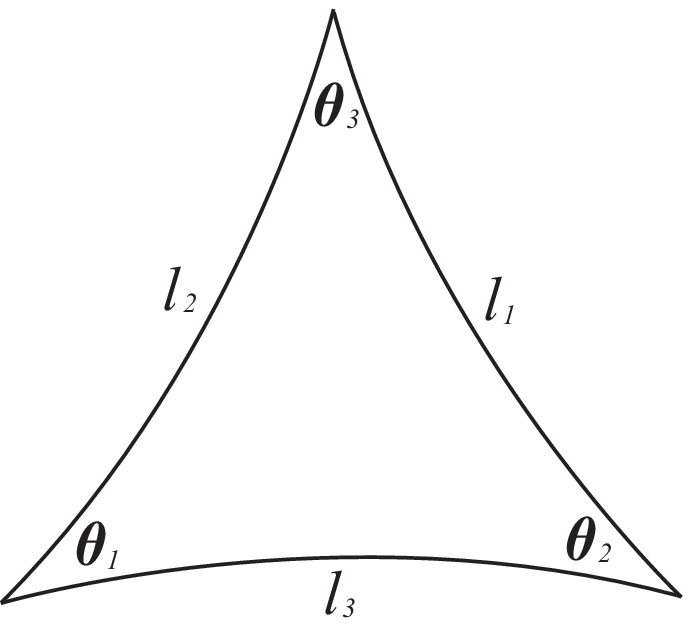}}
\begin{flalign*}
&\text{For}\ \{i,j,k\}=\{1,2,3\}, &\\
&\cosh l_i=\frac{\cos \theta_i+\cos \theta_j\cos
\theta_k}{\sin\theta_j\sin \theta_k}&\\
&\cos \theta_i=\frac{-\cosh l_i +\cosh l_j\cosh l_k}{\sinh
l_j\sinh
l_k}&\\
&\frac{\sin \theta_1}{\sinh l_1}=\frac{\sin \theta_2}{\sinh
l_2}=\frac{\sin \theta_3}{\sinh l_3}&
\end{flalign*}
\begin{picture}(4,2.5)
\put(0,0){\line(1,0){360}}
\end{picture}

\parpic(0cm,0cm)(0cm,4cm)[r]{\includegraphics[scale=.53]{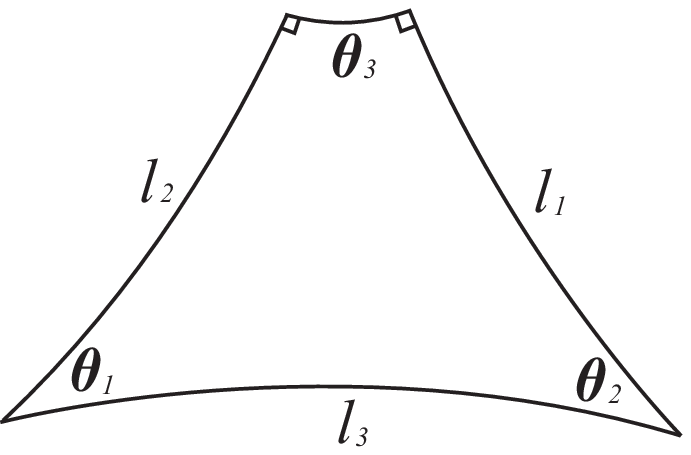}}
\begin{flalign*}
&\text{For}\ \{i,j\}=\{1,2\},&\\
&\sinh l_i=\frac{\cos \theta_i+\cos \theta_j\cosh
\theta_3}{\sin\theta_j\sinh \theta_3}&\\
&\cosh l_3=\frac{\cosh \theta_3+\cos \theta_1\cos \theta_2}{\sin
\theta_1\sin \theta_2}&\\
&\cos \theta_i=\frac{-\sinh l_i +\sinh l_j\cosh l_3}{\cosh
l_j\sinh
l_3}&\\
&\cosh \theta_3=\frac{\cosh l_3 +\sinh l_1\sinh l_2}{\cosh
l_1\cosh
l_2}&\\
&\frac{\sin \theta_1}{\cosh l_1}=\frac{\sin \theta_2}{\cosh
l_2}=\frac{\sinh \theta_3}{\sinh l_3}&
\end{flalign*}
\begin{picture}(4,2.5)
\put(0,0){\line(1,0){360}}
\end{picture}

\parpic(0cm,0cm)(-.8cm,5cm)[r]{\includegraphics[scale=.53]{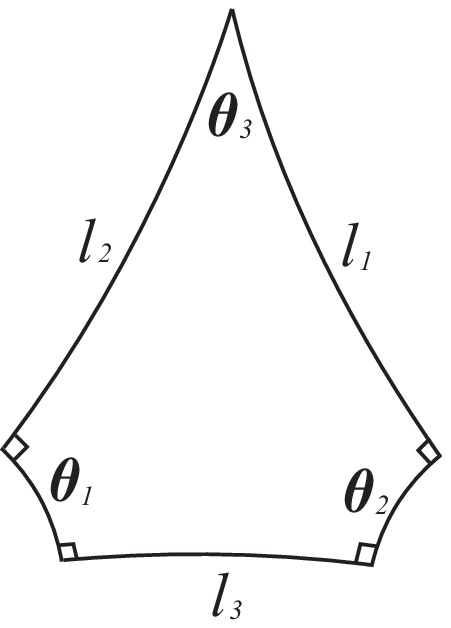}}
\begin{flalign*}
&\text{For}\ \{i,j\}=\{1,2\},&\\
&\sinh l_i=\frac{\cosh \theta_i+\cosh \theta_j\cos \theta_3}{\sinh
\theta_j\sin \theta_3}&\\
&\cosh l_3=\frac{\cos \theta_3+\cosh \theta_1\cosh \theta_2}{\sinh
\theta_1\sinh \theta_2}&\\
&\cos \theta_i=\frac{\sinh l_i +\sinh l_j\cosh l_3}{\cosh l_j\sinh
l_3}&\\
&\cos \theta_3=\frac{-\cosh l_3 +\sinh l_1\sinh l_2}{\cosh
l_1\cosh
l_2}&\\
&\frac{\sinh \theta_1}{\cosh l_1}=\frac{\sinh \theta_2}{\cosh
l_2}=\frac{\sin \theta_3}{\sinh l_3}&
\end{flalign*}
\begin{picture}(4,2.5)
\put(0,0){\line(1,0){360}}
\end{picture}

\parpic(0cm,0cm)(-.8cm,3.8cm)[r]{\includegraphics[scale=.53]{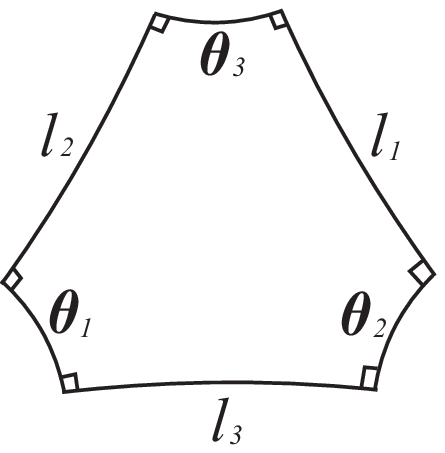}}
\begin{flalign*}
&\text{For}\ \{i,j,k\}=\{1,2,3\},&\\
&\cosh l_i=\frac{\cosh \theta_i+\cosh \theta_j\cosh
\theta_k}{\sinh
\theta_j\sinh \theta_k}&\\
&\cosh \theta_i=\frac{\cosh l_i +\cosh l_j\cosh l_k}{\sinh
l_j\sinh
l_k}&\\
&\frac{\sinh \theta_1}{\sinh l_1}=\frac{\sinh \theta_2}{\sinh
l_2}=\frac{\sinh \theta_3}{\sinh l_3}&
\end{flalign*}
\begin{picture}(4,2.5)
\put(0,0){\line(1,0){360}}
\end{picture}

\parpic(0cm,0cm)(0cm,5cm)[r]{\includegraphics[scale=.53]{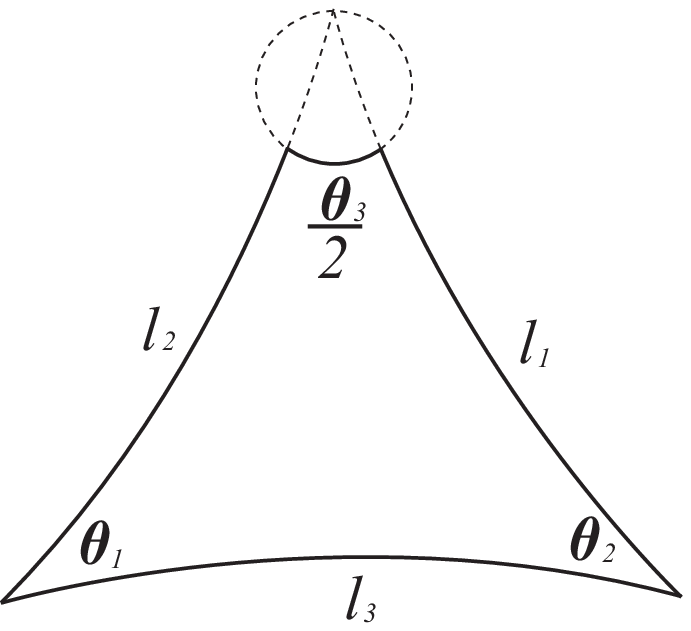}}
\begin{flalign*}
&\text{For}\ \{i,j\}=\{1,2\},&\\
&\frac{e^{l_i}}2=\frac{\cos\theta_i+\cos\theta_j}{\theta_3\sin\theta_j}&\\
&\cosh l_3=\frac{1+\cos \theta_1\cos \theta_2}{\sin \theta_1\sin
\theta_2}&\\
&\cos \theta_i=\frac{-e^{l_i} +e^{l_j}\cosh l_3}{e^{l_j}\sinh l_3}&\\
&\frac{\theta_3^2}{2}=\frac{\cosh l_3 -\cosh
(l_1-l_2)}{\frac{e^{l_1+l_2}}4}&\\
&\frac{\sin \theta_1}{\frac{e^ {l_1}}2}=\frac{\sin
\theta_2}{\frac{e^ {l_2}}2}=\frac{\theta_3}{\sinh l_3}&
\end{flalign*}
\begin{picture}(4,2.5)
\put(0,0){\line(1,0){360}}
\end{picture}

\parpic(0cm,0cm)(-.8cm,5cm)[r]{\includegraphics[scale=.53]{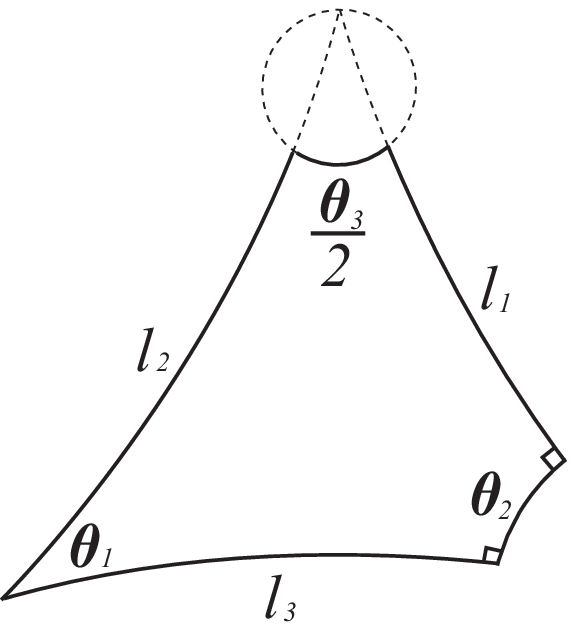}}
\begin{flalign*}
&\frac{e^{l_1}}2=\frac{\cos\theta_1+\cosh\theta_2}{\theta_3\sinh\theta_2}&\\
&\frac{e^{l_2}}2=\frac{\cosh\theta_2+\cos\theta_1}{\theta_3\sin\theta_1}&\\
&\sinh l_3=\frac{1+\cos \theta_1\cosh \theta_2}{\sin \theta_1\sinh
\theta_2}&\\
&\cos \theta_1=\frac{-e^{l_1} +e^{l_2}\sinh l_3}{e^{l_2}\cosh l_3}&\\
&\cosh \theta_2=\frac{e^{l_2} +e^{l_1}\sinh l_3}{e^{l_1}\cosh l_3}&\\
&\frac{\theta_3^2}{2}=\frac{\sinh l_3 +\sinh
(l_2-l_1)}{\frac{e^{l_1+l_2}}4}&\\
&\frac{\sin \theta_1}{\frac{e^ {l_1}}2}=\frac{\sinh
\theta_2}{\frac{e^ {l_2}}2}=\frac{\theta_3}{\cosh l_3}&
\end{flalign*}
\begin{picture}(4,2.5)
\put(0,0){\line(1,0){360}}
\end{picture}

\parpic(0cm,0cm)(-.8cm,5cm)[r]{\includegraphics[scale=.53]{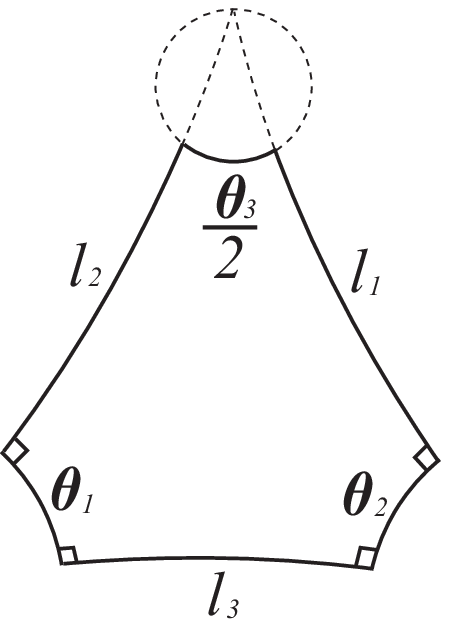}}
\begin{flalign*}
&\text{For}\ \{i,j\}=\{1,2\},&\\
&\frac{e^{l_i}}2=\frac{\cosh\theta_i+\cosh\theta_j}{\theta_3\sinh\theta_j}&\\
&\cosh l_3=\frac{1+\cosh \theta_1\cosh \theta_2}{\sinh
\theta_1\sinh \theta_2}&\\
&\cosh \theta_i=\frac{e^{l_i} +e^{l_j}\cosh l_3}{e^{l_j}\sinh l_3}&\\
&\frac{\theta_3^2}{2}=\frac{\cosh l_3 +\cosh
(l_1-l_2)}{\frac{e^{l_1+l_2}}4}&\\
&\frac{\sinh \theta_1}{\frac{e^ {l_1}}2}=\frac{\sinh
\theta_2}{\frac{e^ {l_2}}2}=\frac{\theta_3}{\sinh l_3}&
\end{flalign*}
\begin{picture}(4,2.5)
\put(0,0){\line(1,0){360}}
\end{picture}

\parpic(0cm,0cm)(0cm,5cm)[r]{\includegraphics[scale=.53]{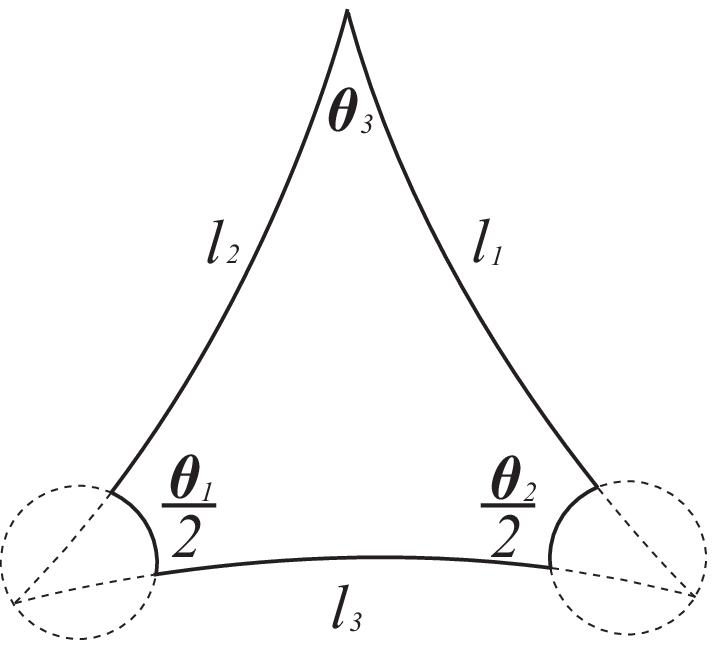}}
\begin{flalign*}
&\text{For}\ \{i,j\}=\{1,2\},&\\
&\frac{e^{l_i}}2=\frac{1+\cos \theta_3}{\theta_j\sin \theta_3}&\\
&\frac{e^{l_3}}2=\frac{1+\cos\theta_3}{\theta_1\theta_2}&\\
&\frac{\theta_i^2}4=\frac{e^{l_i} -e^{l_3-l_j}}{e^{l_j+l_3}}&\\
&\sin^2\frac{\theta_3}{2}=e^{l_3-l_1-l_2}&\\
&\frac{\theta_1}{e^{l_1}}=\frac{\theta_2}{e^{l_2}}=\frac{\sin
\theta_3}{e^{l_3}}&
\end{flalign*}
\begin{picture}(4,2.5)
\put(0,0){\line(1,0){360}}
\end{picture}

\parpic(0cm,0cm)(0cm,4cm)[r]{\includegraphics[scale=.53]{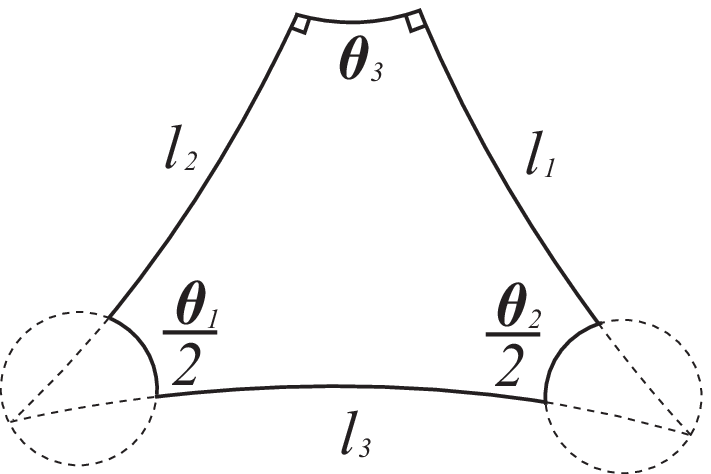}}
\begin{flalign*}
&\text{For}\ \{i,j\}=\{1,2\}, &\\
&\frac{e^{l_i}}2=\frac{1+\cosh \theta_3}{\theta_j\sinh \theta_3} &\\
&\frac{e^{l_3}}2=\frac{1+\cosh\theta_3}{\theta_1\theta_2} &\\
&\frac{\theta_i^2}4=\frac{e^{l_i}+e^{l_3-l_j}}{e^{l_j+l_3}} &\\
&\sinh^2\frac{\theta_3}{2}=e^{l_3-l_1-l_2} &\\
&\frac{\theta_1}{e^{l_1}}=\frac{\theta_2}{e^{l_2}}=\frac{\sinh
\theta_3}{e^{l_3}} &
\end{flalign*}
\begin{picture}(4,2.5)
\put(0,0){\line(1,0){360}}
\end{picture}

\parpic(0cm,0cm)(0cm,4cm)[r]{\includegraphics[scale=.53]{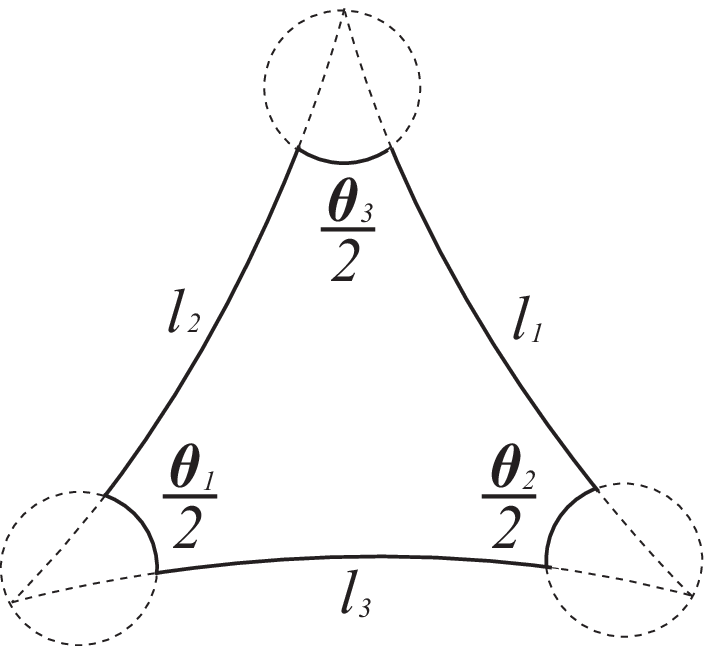}}
\begin{flalign*}
&\text{For}\ \{i,j,k\}=\{1,2,3\}, &\\
&\frac{e^{l_i}}2=\frac2{\theta_j\theta_k}  &\\
&\frac{\theta_i^2}4=e^{l_i-l_j-l_k} &\\
&\frac{\theta_1}{e^{l_1}}=\frac{\theta_2}{e^{l_2}}=\frac{\theta_3}{e^{l_3}}
&
\end{flalign*}
\newpage

\section*{Appendix B. A proof of Lemma \ref{thm:determinant} and Lemma \ref{thm:derivative-cosine}}

For simplicity, we abuse the notation. Let
$$g(\theta_i):=\rho_{\varepsilon_i}(\theta_i),\ \ g'(\theta_i):=\rho'_{\varepsilon_i}(\theta_i),$$
$$f(l_k):=\tau_{\varepsilon_i\varepsilon_j}(l_k),\ \ f'(l_k):=\tau'_{\varepsilon_i\varepsilon_j}(l_k)$$
for $\{i,j,k\}=\{1,2,3\}$, where $g,f$ are not well-defined
functions which depend on the type of generalized vertices.

In this simplified notation we have
\begin{align*}
-\det G_l=&\det \left(
\begin{array}{ccc}
\varepsilon_1&f'(l_3)&f'(l_2) \\
f'(l_3)&\varepsilon_2&f'(l_1) \\
f'(l_2)&f'(l_1)&\varepsilon_3
\end{array}
\right)\\=&\varepsilon_1\varepsilon_2\varepsilon_3+2f'(l_1)f'(l_2)f'(l_3)
-\varepsilon_1f'^2(l_1)-\varepsilon_2f'^2(l_2)-\varepsilon_3f'^2(l_3).
\end{align*}
If $\varepsilon_3=0,$ then $g(\theta_3)=\theta_3, f(l_1)=\frac
{e^{l_1}}2, f(l_2)=\frac {e^{l_2}}2.$ By the cosine law
(\ref{fml:cosine2}), we have
$$\frac{\theta_3^2}2=\frac{f'(l_3)-\frac12\varepsilon_1e^{l_1-l_2}-\frac12\varepsilon_2e^{l_2-l_1}}
{f(l_1)f(l_2)}.$$ Thus negative of the right hand side of
(\ref{fml:det G_l}) is
\begin{align*}
f^2(l_1)f^2(l_2)g^2(\theta_3)
&=f^2(l_1)f^2(l_2)\theta_3^2\\
&=f^2(l_1)f^2(l_2)\frac{2f'(l_3)-\varepsilon_1e^{l_1-l_2}-\varepsilon_2e^{l_2-l_1}}
{f(l_1)f(l_2)}\\
&=f(l_1)f(l_2)(2f'(l_3)-\varepsilon_1e^{l_1-l_2}-\varepsilon_2e^{l_2-l_1})\\
&=\frac {e^{l_1}}2\frac
{e^{l_2}}2(2f'(l_3)-\varepsilon_1e^{l_1-l_2}-\varepsilon_2e^{l_2-l_1})\\
&=2\frac{e^{l_1}}2\frac{e^{l_2}}2f'(l_3)
-\varepsilon_1\frac{e^{2l_1}}4 -\varepsilon_2\frac{e^{2l_2}}4 \\
&=-\det G_l.
\end{align*}

If $\varepsilon_3=\pm1,$ then
$g^2(\theta_3)=\varepsilon_3(1-g'^2(\theta_3)).$ And we have
\begin{align}\label{fml:f}
f^2(l_i)-f'^2(l_i)=(\frac12e^l-\frac12\varepsilon_j\varepsilon_ke^{-l})^2-
(\frac12e^l+\frac12\varepsilon_j\varepsilon_ke^{-l})^2=-\varepsilon_j\varepsilon_k.
\end{align}

Thus negative of the right hand side of (\ref{fml:det G_l}) is
\begin{align*}
f^2(l_1)f^2(l_2)g^2(\theta_3)
&=f^2(l_1)f^2(l_2)(1-g'^2(\theta_3))\\
&=\varepsilon_3(f^2(l_1)f^2(l_2)-(-\varepsilon_3f'(l_3)+f'(l_1)f'(l_2))^2)   \\
&=\varepsilon_3((f'^2(l_1)-\varepsilon_2\varepsilon_3)(f'^2(l_2)-\varepsilon_3\varepsilon_1)
-(-\varepsilon_3f'(l_3)+f'(l_1)f'(l_2))^2)\\
&=\varepsilon_3(\varepsilon_1\varepsilon_2\varepsilon_3^2
+2\varepsilon_3f'(l_1)f'(l_2)f'(l_3)
-\varepsilon_1\varepsilon_3f'^2(l_1)-\varepsilon_2\varepsilon_3f'^2(l_2)-\varepsilon_3^2f'^2(l_3))\\
&=-\det G_l.
\end{align*}
The second equality is due to the cosine law, the third equality
is due to (\ref{fml:f}) and the last equality is due to
$\varepsilon_3=\pm1$.

And we have
\begin{align*} -\det G_\theta=&\det \left(
\begin{array}{ccc}
-1&g'(\theta_3)&g'(\theta_2) \\
g'(\theta_3)&-1&g'(\theta_1) \\
g'(\theta_2)&g'(\theta_1)&-1
\end{array}
\right)\\=&-1+2g'(\theta_1)g'(\theta_2)g'(\theta_3)+g'^2(\theta_1)+g'^2(\theta_2)+g'^2(\theta_3).
\end{align*}

Notice that we have
\begin{align}\label{fml:g}
g'^2(\theta_i)+\varepsilon_ig^2(\theta_i)=1.
\end{align} Thus negative of the right hand side of (\ref{fml:det G_a}) is
\begin{align*}
g^2(\theta_1)g^2(\theta_2)f^2(l_3)&=g^2(\theta_1)g^2(\theta_2)(f'^2(l_3)-\varepsilon_1\varepsilon_2)\\
&=(g'(\theta_3)+g'(\theta_1)g'(\theta_2))^2-g^2(\theta_1)g^2(\theta_2)\varepsilon_1\varepsilon_2\\
&=(g'(\theta_3)+g'(\theta_1)g'(\theta_2))^2-(1-g'^2(\theta_1))(1-g'^2(\theta_2))\\
&=-\det G_\theta.
\end{align*}
The second equality is due to the cosine law and the third
equality is due to (\ref{fml:g}).

We can prove either one of the derivative cosine law
(\ref{fml:derivative1}) and (\ref{fml:derivative2}). The other one
will be a corollary duo to Lemma \ref{thm:inverse matrix}. For
example, we give a proof of (\ref{fml:derivative1}).

By the cosine law (\ref{fml:cosine1}) we have
$$f'(l_i)g(\theta_j)g(\theta_k)=g'(\theta_i)+g'(\theta_j)g'(\theta_k).$$
After differentiating the two sides we have
\begin{multline*}
f''(l_i)g(\theta_j)g(\theta_k)dl_i+f'(l_i)g'(\theta_j)g(\theta_k)d\theta_j+f'(l_i)g(\theta_j)g'(\theta_k)d\theta_k\\
=g''(\theta_i)d\theta_i+g''(\theta_j)g'(\theta_k)d\theta_j+g'(\theta_j)g''(\theta_k)d\theta_k
\end{multline*} which is equivalent to
\begin{multline}\label{fml:appendix B1}
f''(l_i)g(\theta_j)g(\theta_k)dl_i\\
=g''(\theta_i)d\theta_i+[g''(\theta_j)g'(\theta_k)-f'(l_i)g'(\theta_j)g(\theta_k)]d\theta_j\\
+[g'(\theta_j)g''(\theta_k)-f'(l_i)g(\theta_j)g'(\theta_k)]d\theta_k.
\end{multline}
By the cosine law (\ref{fml:cosine1}), the coefficient of
$d\theta_j$ in (\ref{fml:appendix B1}) is
\begin{align}
&g''(\theta_j)g'(\theta_k)-f'(l_i)g'(\theta_j)g(\theta_k) \notag \\
&=g''(\theta_j)g'(\theta_k)-\frac{g'(\theta_i)+g'(\theta_j)g'(\theta_k)}{g(\theta_j)g(\theta_k)}g'(\theta_j)g(\theta_k) \notag\\
&=\frac1{g(\theta_j)}[(g(\theta_j)g''(\theta_j)-g'^2(\theta_j))g'(\theta_k)-g'(\theta_i)g'(\theta_j)].\label{fml:appendix
B2}
\end{align}
For $g(\theta)=\sin\theta,$ or $\sinh\theta,$ or $\theta,$ we
always have $$g(\theta_j)g''(\theta_j)-g'^2(\theta_j)=-1.$$ Hence
(\ref{fml:appendix B2}) is
$\frac1{g(\theta_j)}(-g'(\theta_k)-g'(\theta_i)g'(\theta_j))=-g(\theta_i)f(l_k)$
due to the cosine law (\ref{fml:cosine1}). By symmetry, the
similar formula holds for the coefficient of $d\theta_k.$ Hence
(\ref{fml:appendix B1}) is
$$f''(l_i)g(\theta_j)g(\theta_k)dl_i=g''(\theta_i)d\theta_i-g(\theta_i)f(l_k)d\theta_j-g(\theta_i)f(l_j)d\theta_k.$$
By the definition of $f,$ we have $f''=f.$ Thus
\begin{align*}
dl_i&=\frac{-g(\theta_i)}{f(l_i)g(\theta_j)g(\theta_k)}
(-\frac{g''(\theta_i)}{g(\theta_i)}d\theta_i+f(l_k)d\theta_j+f(l_j)d\theta_k)\\
&=\frac{-g(\theta_i)}{f(l_i)g(\theta_j)g(\theta_k)}
(\varepsilon_id\theta_i+f(l_k)d\theta_j+f(l_j)d\theta_k).
\end{align*}

This proves (\ref{fml:derivative1}).

\bibliographystyle{amsplain}

\end{document}